\documentclass{amsart}

\usepackage{amssymb}
\usepackage{amsmath}
\usepackage{amsthm}
\usepackage{amsfonts}

\usepackage{a4wide}
\usepackage{longtable}
\usepackage{multirow}
\usepackage{setspace}

\newtheorem{theorem}{Theorem}[section]
\newtheorem{proposition}{Proposition}[section]
\newtheorem{corollary}{Corollary}[section]

\theoremstyle{definition}
\newtheorem{definition}{Definition}[section]
\newtheorem{remark}{Remark}[section]
\newtheorem{example}{Example}[section]

\numberwithin{equation}{section}

\begin{document}

\title{Three stage towers of \(5\)-class fields}

\author{Daniel C. Mayer}
\address{Naglergasse 53\\8010 Graz\\Austria}
\email{algebraic.number.theory@algebra.at}
\urladdr{http://www.algebra.at}
\thanks{Research supported by the Austrian Science Fund (FWF): P 26008-N25}

\subjclass[2000]{Primary 11R37, 11R29, 11R11, 11R20; Secondary 20D15, 20--04}
\keywords{\(5\)-class towers, \(5\)-class groups, \(5\)-capitulation, real quadratic fields, dihedral fields of degree \(10\);
finite \(5\)-groups with two generators, descendant tree, \(p\)-group generation algorithm,
nuclear rank, bifurcation, \(p\)-multiplicator rank, relation rank, generator rank, Shafarevich cover,
Artin transfers, partial order of Artin patterns}

\date{April 23, 2016}

\begin{abstract}
With \(K=\mathbb{Q}\left(\sqrt{3\,812\,377}\right)\)
we give the first example of an algebraic number field
possessing a \(5\)-class tower of exact length \(\ell_5{K}=3\).
The rigorous proof is conducted by means of the \(p\)-group generation algorithm,
showing the existence of a unique finite metabelian \(5\)-group \(\mathfrak{G}\)
with abelianization \(\lbrack 5,5\rbrack\)
having the kernels \((M_1,\mathfrak{G}^5)\) and targets \(\left(\lbrack 25,5,5,5\rbrack,\lbrack 5,5\rbrack^5\right)\)
of Artin transfers \(T_i:\,\mathfrak{G}\to M_i/M_i^\prime\) to its six maximal subgroups \(M_i\),
prescribed by arithmetical invariants of \(K\).
Thus, \(\mathfrak{G}\) must be the second \(5\)-class group \(\mathrm{G}_5^2{K}\) of the real quadratic field \(K\)
but cannot be its \(5\)-class tower group \(\mathrm{G}_5^\infty{K}\),
since the relation rank \(d_2\mathfrak{G}=4\) is too big.
We provide evidence of exactly five non-isomorphic extensions \(G\) of \(\mathfrak{G}\)
having the required relation rank \(d_2 G=3\) and derived length \(\mathrm{dl}(G)=3\)
whose metabelianization \(G/G^{\prime\prime}\) is isomorphic to \(\mathfrak{G}\).
Consequently, \(\mathrm{G}_5^\infty{K}\) must be one of the five non-metabelian groups \(G\). 
\end{abstract}

\maketitle



\section{Introduction}
\label{s:Intro}

Given a prime number \(p\ge 2\),
a positive integer \(n\ge 1\),
and an algebraic number field \(K\),
the \(n\)th \textit{Hilbert \(p\)-class field} \(\mathrm{F}_p^n{K}\) of \(K\) is
the maximal unramified Galois extension of \(K\) whose automorphism group
\(\mathrm{G}_p^n{K}:=\mathrm{Gal}\left(\mathrm{F}_p^n{K}\vert K\right)\)
is a finite \(p\)-group of derived length \(\mathrm{dl}(\mathrm{G}_p^n{K})\le n\).
In particular, the first Hilbert \(p\)-class field \(\mathrm{F}_p^1{K}\) of \(K\)
is the maximal abelian unramified \(p\)-extension of \(K\).
Its Galois group \(\mathrm{G}_p^1{K}\) is isomorphic to the \(p\)-class group
\(\mathrm{Cl}_p{K}:=\mathrm{Syl}_p\mathrm{Cl}(K)\) of \(K\),
according to Artin's reciprocity law of class field theory
\cite{Ar1}.

The projective limit \(\varprojlim_{n\ge 1}\mathrm{G}_p^n{K}\)
of all \(n\)th \(p\)-\textit{class groups} \(\mathrm{G}_p^n{K}\) of \(K\)
is isomorphic to the possibly infinite topological Galois group \(G:=\mathrm{G}_p^\infty{K}\)
of the maximal unramified pro-\(p\) extension
\(\mathrm{F}_p^\infty{K}:=\bigcup_{n\ge 1}\,\mathrm{F}_p^n{K}\),
that is the \textit{Hilbert \(p\)-class tower}, of \(K\).
Conversely, the \(n\)th \(p\)-class group can be obtained as the \(n\)th derived quotient
\(G/G^{(n)}\simeq\mathrm{G}_p^n{K}\) of the \(p\)-\textit{class tower group} \(G\) of \(K\).

The \textit{length} of the \(p\)-class tower of \(K\) is defined as the derived length of \(G\),

\begin{equation}
\label{eqn:TowerLength}
\ell_p{K}:=\mathrm{dl}\left(\mathrm{G}_p^\infty{K}\right).
\end{equation}

For each \(n\ge 2\), inclusively \(n=\infty\), the commutator subgroup of \(\mathrm{G}_p^n{K}\)
is isomorphic to \(\mathrm{Gal}\left(\mathrm{F}_p^n{K}\vert\mathrm{F}_p^1{K}\right)\), and the abelianization of \(\mathrm{G}_p^n{K}\)
is isomorphic to the \(p\)-class group \(\mathrm{Cl}_p{K}\) of \(K\),

\begin{equation}
\label{eqn:GeneratorRank}
\mathrm{G}_p^n{K}/\left(\mathrm{G}_p^n{K}\right)^\prime\simeq\mathrm{G}_p^1{K}\simeq\mathrm{Cl}_p{K},
\end{equation}

\noindent
in particular, the generator rank
\(d_1\left(\mathrm{G}_p^n{K}\right):=\dim_{\mathbb{F}_p}\mathrm{H}^1(\mathrm{G}_p^n{K},\mathbb{F}_p)\)
of \(\mathrm{G}_p^n{K}\) coincides with the \(p\)-class rank
\(r_p{K}:=\dim_{\mathbb{F}_p}\left(\mathrm{Cl}_p{K}\bigotimes_{\mathbb{Z}_p}\mathbb{F}_p\right)\)
of \(K\), according to Burnside's basis theorem.


After the preceding clarification of the technical framework for this article,
we devote a minimum of space to an outline of the historical evolution of finding \(p\)-class towers with increasing finite length.
Thereby, we use the notation of the SmallGroups database
\cite{BEO1},
\cite{BEO2},
and partially give more details than the original papers of the cited authors,

We have \(\ell_p{K}=0\) if and only if the class number \(\#\mathrm{Cl}(K)\) of the field \(K\) is not divisible by \(p\).

If the \(p\)-class group \(\mathrm{Cl}_p{K}>1\) is cyclic, then we have a single stage tower with \(\ell_p{K}=1\),
by Formula
\eqref{eqn:GeneratorRank}.
If \(K=\mathbb{Q}\left(\sqrt{d}\right)\) is a quadratic field and \(p\ge 3\) is odd, the converse is also true
\cite[Thm. 4.1, (1), p. 486]{Ma1}.
However, for the discriminant \(d=-84=-2^2\cdot 3\cdot 7\),
we have \(\mathrm{Cl}_2{K}\simeq\lbrack 2,2\rbrack\) but nevertheless only \(\ell_2{K}=1\)
\cite[(ii), p. 277]{Ki},
\cite[\S\ 9, pp. 501--503]{Ma1}.

In \(1930\), Hasse
\cite[\S\ 27, pp. 173--174]{Ha2}
gave the first example of \(\ell_2{K}\ge 2\) for a field of type \(\lbrack 4,2\rbrack\),
due to private communications by Furtw\"angler and Artin.
It has discriminant \(d=-260=-2^2\cdot 5\cdot 13\), \(\mathrm{G}_2^2{K}\simeq\langle 2^4,6\rangle\),
and in fact \(\ell_2{K}=2\).
A field of type \(\lbrack 2,2\rbrack\) with \(\ell_2{K}=2\) is given by
\(d=-120=-2^3\cdot 3\cdot 5\) with \(\left(\frac{5}{3}\right)=-1\),
and has \(\mathrm{G}_2^2{K}\simeq\langle 2^3,4\rangle\) the quaternion group
\cite[(vi.a,b), p. 278]{Ki}.
However, a field of type \(\lbrack 2,2\rbrack\) can never have a \(2\)-class tower of length bigger than two.

Consequently, the first field with \(\ell_2{K}=3\) was of type \(\lbrack 4,2\rbrack\).
It was discovered in \(2003\) by Bush
\cite[Prop. 2, p. 321]{Bu}
and has discriminant \(d=-1\,780=-2^2\cdot 5\cdot 89\).
Its \(2\)-tower group is one of two candidates
\(\langle 2^8,426\vert 427\rangle\)
with class \(5\) and coclass \(3\).
In \(2006\), Nover provided evidence of \(\ell_2{K}=3\) for a field of type \(\lbrack 2,2,2\rbrack\)
with \(d=-3\,135=-3\cdot 5\cdot 11\cdot 19\)
and four candidates for \(\mathrm{G}_2^3{K}\) of order \(2^{13}\), class \(7\) and coclass \(6\)
\cite[\S\ 3.4, pp. 7--8]{BoNo},
\cite[Prop. 5, p. 239]{No}.

Now we turn to odd primes \(p\ge 3\).
In \(1934\), Scholz and Taussky
\cite[\S\ 3, pp. 39--41]{SoTa}
manually calculated the first occurrences of \(\ell_3{K}=2\)
for \(d=-4\,027\) of type \(\lbrack 3,3\rbrack\),
resp. \(d=-3\,299\) of type \(\lbrack 9,3\rbrack\).
The corresponding \(3\)-class tower groups are the Schur \(\sigma\)-groups
\(\langle 3^5,5\rangle\), resp. one of the two candidates \(\langle 3^6,14\vert 15\rangle\).
However, they also claimed a two stage tower for \(d=-9\,748=2^2\cdot 2\,437\)
\cite[\S\ 3, p. 41]{SoTa},
where neither of the two candidates \(\langle 3^7,302\vert 306\rangle\)
for the second \(3\)-class group is a Schur \(\sigma\)-group
\cite{KoVe}.
This claim was corrected by Bush and ourselves
\cite[Cor. 4.1.1, p. 775]{BuMa}
in \(2015\). We proved the first sufficient criterion for
complex quadratic fields \(K\) with \(\ell_3{K}=3\),
and provided two candidates \(\mathrm{G}_3^3{K}\simeq\langle 3^6,54\rangle-\#2;2\vert 6\)
for \(d=-9\,748\), in the notation of the ANUPQ package
\cite{GNO}.
These were the first non-metabelian Schur \(\sigma\)-groups.
They have order \(3^8\), class \(5\), coclass \(3\), derived length \(3\),
and are discussed in more detail in
\cite[(70), p. 190]{Ma6},
\cite[Fig. 10, p. 191]{Ma6},
and
\cite[\S\S\ 6--7, p. 751--756]{Ma8}.

The first occurrences of \(\ell_5{K}=2\), resp. \(\ell_7{K}=2\), were discovered in \(2013\) by ourselves
for \(d=-11\,199\)
\cite[Tbl. 3.13, p. 450]{Ma4},
resp. \(d=-63\,499\)
\cite[Tbl. 3.14, p. 450]{Ma4}.
The corresponding groups \(\mathrm{G}_5^2{K}\simeq\langle 5^5,12\rangle\), resp.
\(\mathrm{G}_7^2{K}\simeq\langle 7^5,14\vert 16\rangle\) with two candidates,
are metabelian Schur \(\sigma\)-groups.
Our attempt to prove \(\ell_5{K}=3\) for \(d=-62\,632\)
with two candidates \(\mathrm{G}_5^2{K}\simeq\langle 5^6,564\rangle-\#1;2\vert 3\),
resp. \(d=-67\,063\)
with six candidates \(\mathrm{G}_5^2{K}\simeq\langle 5^6,564\rangle-\#1;8\vert 9\vert 11\vert 12\vert 14\vert 15\),
was not successful, since we could not find suitable Schur \(\sigma\)-groups \(\mathrm{G}_5^3{K}\).

Little progress was achieved for real quadratic fields
until we developed new techniques for determining their \(p\)-class tower in
\cite{Ma7}.
We found \(\ell_3{K}=2\) for \(d=422\,573\) with \(\mathrm{G}_3^2{K}\simeq\langle 3^5,5\rangle\) 
\cite[Tbl. 4, p. 498]{Ma1},
and \(\ell_5{K}=2\) for \(d=4\,954\,652\) with \(\mathrm{G}_5^2{K}\simeq\langle 5^5,12\rangle\).
With extensive computational effort we succeeded in proving \(\ell_3{K}=3\)
for \(d=957\,013\), which has \(\mathrm{G}_3^3{K}\simeq\langle 3^7,273\rangle\)
with class \(5\), coclass \(2\), and relation rank \(3\)
\cite[Exm. 6.3, pp. 306--307]{Ma7}.
However, the attempt to prove that the quintic analogue \(d=12\,562\,849\) has \(\ell_5{K}=3\)
with either of the two candidates \(\langle 5^6,680\rangle-\#1;1\vert 2\) for \(\mathrm{G}_5^3{K}\),
having order \(5^7\), class \(5\) and coclass \(2\),
failed, due to a complete exhaustion of RAM
during the construction of two successive unramified cyclic quintic extensions of the real quadratic field \(K\),
that is an extension of absolute degree \(50\).


Exploring \(p\)-groups of coclass \(1\) instead of coclass \(2\),
which would be useless for finding \(3\)-class towers of three stages,
turned out to be the crucial idea how to circumvent such hopeless number theoretic computations of high complexity
and to conduct a purely group theoretic proof
of the principal results of this article, which are the following statements, in a succinct coarse form.

\begin{theorem}
\label{thm:MainTheorem}

Let \(K\) be a real quadratic field with \(5\)-class group \(\mathrm{Cl}_5{K}\) of type \(\lbrack 5,5\rbrack\)
and denote by \(L_1,\ldots,L_6\) its six unramified cyclic quintic extensions.
If \(K\) possesses the \(5\)-capitulation type

\begin{equation}
\label{eqn:TKT}
\varkappa(K)\sim (1,0,0,0,0,0), \text{ with fixed point } 1,
\end{equation}

\noindent
in the six extensions \(L_i\), and if the \(5\)-class groups \(\mathrm{Cl}_5{L_i}\)
are given by

\begin{equation}
\label{eqn:TTT}
\tau(K)\sim\left(\lbrack 25,5,5,5\rbrack,\lbrack 5,5\rbrack,\lbrack 5,5\rbrack,\lbrack 5,5\rbrack,\lbrack 5,5\rbrack,\lbrack 5,5\rbrack\right),
\end{equation}

\noindent
then the \(5\)-class tower
\(K<\mathrm{F}_5^1{K}<\mathrm{F}_5^2{K}<\mathrm{F}_5^3{K}=\mathrm{F}_5^\infty{K}\)
of \(K\) has exact length \(\ell_5{K}=3\).

\end{theorem}


\begin{corollary}
\label{cor:MainTheorem}
A real quadratic field \(K\) which satisfies the assumptions in Theorem
\ref{thm:MainTheorem}, in particular the Formulas
\eqref{eqn:TKT}
and
\eqref{eqn:TTT},
has the unique second \(5\)-class group

\begin{equation}
\label{eqn:Second5ClassGroup}
\mathrm{G}_5^2{K}\simeq\langle 15625,635\rangle
\end{equation}

\noindent
with order \(5^6\), class \(5\), coclass \(1\), derived length \(2\), and relation rank \(4\),
and one of the following five candidates for the \(5\)-class tower group

\begin{equation}
\label{eqn:5TowerGroup}
\mathrm{G}_5^\infty{K}=\mathrm{G}_5^3{K}\simeq\langle 78125,n\rangle, \quad \text{ where } n\in\lbrace 361,373,374,385,386\rbrace,
\end{equation}

\noindent
with order \(5^7\), class \(5\), coclass \(2\), derived length \(3\), and relation rank \(3\).

\end{corollary}


\begin{example}
\label{exm:Realization}
The minimal fundamental discriminant \(d\) of a real quadratic field \(K=\mathbb{Q}(\sqrt{d})\)
satisfying the conditions
\eqref{eqn:TKT}
and
\eqref{eqn:TTT}
is given by

\begin{equation}
\label{eqn:Realization}
d=3\,812\,377=991\cdot 3\,847.
\end{equation}

\noindent
The next occurrence is \(d=19\,621\,905=3\cdot 5\cdot 307\cdot 4\,261\), which is currently the biggest known example,
whereas \(d=21\,281\,673=3\cdot 7\cdot 53\cdot 19\,121\) satisfies
\eqref{eqn:TTT}
but has a different \(\varkappa(K)\sim (2,0,0,0,0,0)\), without fixed point.

\end{example}


The organization of this paper is as follows.
In \S\
\ref{s:PatternRecognition},
we show how the arithmetical invariants \(\varkappa(K)\) and \(\tau(K)\)
of an algebraic number field \(K\)
(e.g., the data in Formula
\eqref{eqn:TKT}
and
\eqref{eqn:TTT})
can be translated into group theoretic information
\(\varkappa(\mathfrak{G})\) and \(\tau(\mathfrak{G})\)
on the second \(p\)-class group \(\mathfrak{G}=\mathrm{G}_p^2{K}\) of \(K\),
following the ideas indicated in Artin's paper
\cite{Ar2},
which led to Furtw\"angler's proof
\cite{Fw}
of the famous Principal Ideal Theorem
and were expanded in detail by Hasse
\cite{Ha2}.

\S\
\ref{s:MonotonyDescendantTrees}
is devoted to the study of the interplay
between the Artin pattern
\(\mathrm{AP}(G)=\left(\tau(G),\varkappa(G)\right)\)
of finite \(p\)-groups \(G\) on a descendant tree \(\mathcal{T}(R)\) with root \(R\)
and the partial order \(G>\pi{G}\) induced by the (child, parent)-pairs \((G,\pi{G})\)
of the tree, which was discovered recently in
\cite{Ma11}.
These monotony properties of \(\tau(G)\) and \(\varkappa(G)\) are employed
in the \(p\)-group generation algorithm by Newman
\cite{Nm1}
and O'Brien
\cite{Ob}
for identifying the unique metabelianization
\(\mathfrak{G}=\mathrm{G}_5^2{K}=G/G^{\prime\prime}\simeq\langle 15625,635\rangle\)
of the \(5\)-class tower group \(G=\mathrm{G}_5^\infty{K}\) of \(K\)
by the strategy of pattern recognition via Artin transfers in \S\ 
\ref{s:Proofs}.
An unexpected bifurcation in the tree \(\mathcal{T}\left(\langle 25,2\rangle\right)\) enabled the discovery
of five non-isomorphic possibilities
\(G\simeq\langle 78125,n\rangle\) with \(n\in\lbrace 361,373,374,385,386\rbrace\)
and derived length \(\mathrm{dl}(G)=3\) for \(G\) itself.

The termination condition for the algorithm
is expressed by the growth of the \(\tau(G)\)-component beyond a break-off bound.
At this stage it is still open whether the length of the \(5\)-class tower of \(K\)
is \(\ell_5{K}=3\) or \(\ell_5{K}=2\),
since only group theoretic information was used up to this point.

Therefore, 
arithmetical constraints for \(G\) in form of the relation rank,
which has been determined by Shafarevich
\cite{Sh},
come into the play in \S\S\
\ref{s:RelationRank}
and
\ref{s:ShafarevichCover},
and eliminate the possibility of a metabelian tower with length \(\ell_5{K}=2\)
for a real quadratic field \(K\).



\section{Pattern recognition via Artin transfers}
\label{s:PatternRecognition}

\subsection{Translation from number theory to group theory}
\label{ss:TranslationFieldGroup}

We start with an arithmetical situation where Artin patterns play an important role.
Let \(p\ge 2\) be a prime number.
Suppose we are given an algebraic number field \(K\vert\mathbb{Q}\)
with non-trivial \(p\)-class group \(\mathrm{Cl}_p{K}>1\).
Then the intermediate fields \(L\) between \(K\) and its Hilbert \(p\)-class field \(\mathrm{F}_p^1{K}\)
are exactly the (finitely many) abelian unramified extensions of \(K\) with degree a power of \(p\).
For each of them, we denote by \(j_{L\vert K}:\,\mathrm{Cl}_p{K}\to\mathrm{Cl}_p{L}\)
the \textit{class extension} homomorphism.

\begin{definition}
\label{dfn:FieldArtinPattern}
By the \textit{ restricted Artin pattern} of \(K\) we understand the pair consisting of
the family \(\tau(K)\) of the \(p\)-class groups
of all extensions \(K<L<\mathrm{F}_p^1{K}\) as its first component
and the \(p\)-\textit{capitulation type} \(\varkappa(K)\) as its second component,

\begin{equation}
\label{eqn:FieldArtinPattern}
\mathrm{AP}(K):=\left(\tau(K),\varkappa(K)\right),
\quad \tau(K):=\left(\mathrm{Cl}_p{L}\right)_{K<L<\mathrm{F}_p^1{K}},
\quad \varkappa(K):=\left(\ker{j_{L\vert K}}\right)_{K<L<\mathrm{F}_p^1{K}}.
\end{equation}

\end{definition}

\begin{remark}
\label{rmk:FieldArtinPattern}
It is convenient to replace the group objects in the unordered family
\(\tau(K)\), resp. \(\varkappa(K)\),
by ordered abelian type invariants, resp. ordered numerical identifiers,
accepting the drawback of dependence on the selected ordering
(indicated by using the symbol \(\sim\) instead of \(=\)).

In the case \(p=5\) and \(C:=\mathrm{Cl}_5{K}\simeq\lbrack 5,5\rbrack\), which is relevant for this paper,
let \(L_1,\ldots,L_6\) be the six intermediate fields between \(K\) and \(\mathrm{F}_5^1{K}\)
and denote by \(N_i:=\mathrm{Norm}_{L_i\vert K}\left(\mathrm{Cl}_5{L_i}\right)\) their norm class groups,
which are exactly the six cyclic subgroups of index \(5\) in \(C\).
In this sense, the notation of the \(5\)-capitulation type in Formula
\eqref{eqn:TKT}
(where \(0\) means a total capitulation) is an abbreviation for

\begin{equation}
\label{eqn:CapitulationType}
\varkappa(K)\sim\left(N_1,C,C,C,C,C\right).
\end{equation}

\end{remark}


Now we turn to the group theoretic aspect of Artin patterns.
Let \(G\) be a finite \(p\)-group
or an infinite topological pro-\(p\) group with finite abelianization \(G/G^\prime\).
For each of the (finitely many) intermediate groups \(M\)
between the commutator subgroup \(G^\prime\) and the group \(G\),
we denote by \(T_{G,M}:\,G\to M/M^\prime\) the \textit{Artin transfer} homomorphism
from \(G\) to the abelianization \(M/M^\prime\) of \(M\)
\cite[Dfn. 3.1, p. 5]{Ma6}.

\begin{definition}
\label{dfn:GroupArtinPattern}
By the \textit{ restricted Artin pattern} of \(G\) we understand the pair consisting of
the \textit{transfer target type} (TTT) \(\tau(G)\) as its first component
and the \textit{transfer kernel type} (TKT) \(\varkappa(G)\) as its second component,

\begin{equation}
\label{eqn:GroupArtinPattern}
\mathrm{AP}(G):=\left(\tau(G),\varkappa(G)\right),
\quad \tau(G):=\left(M/M^\prime\right)_{G^\prime<M<G},
\quad \varkappa(G):=\left(\ker{T_{G,M}}\right)_{G^\prime<M<G}.
\end{equation}

\end{definition}

\begin{remark}
\label{rmk:GroupArtinPattern}
Again, we usually replace the group objects in the unordered family
\(\tau(G)\), resp. \(\varkappa(G)\),
by ordered abelian type invariants, resp. ordered numerical identifiers.

In the case \(p=5\) and \(G/G^\prime\simeq\lbrack 5,5\rbrack\), which will be of concern in this paper,
we denote by \(M_1,\ldots,M_6\) the six intermediate groups between \(G^\prime\) and \(G\)
and we write the transfer kernel type corresponding to Formula
\eqref{eqn:TKT}
as

\begin{equation}
\label{eqn:TransferKernelType}
\varkappa(G)\sim\left(M_1,G,G,G,G,G\right) \quad \text{ or briefly } \varkappa(G)\sim\left(1,0,0,0,0,0\right).
\end{equation}

\end{remark}


Finally, we combine the number theoretic and group theoretic view of Artin patterns.

\begin{theorem}
\label{thm:ArtinPattern}
Let \(p\) be a prime number.
Assume that \(K\) is a number field, and let
\(\mathfrak{G}:=\mathrm{G}_p^2{K}\) be the second \(p\)-class group of \(K\).
Then \(\mathfrak{G}\) and \(K\) share a common restricted Artin pattern,

\begin{equation}
\label{eqn:ArtinPattern}
\mathrm{AP}(\mathfrak{G})=\mathrm{AP}(K), \quad \text{ that is } \quad \tau(\mathfrak{G})=\tau(K)\text{ and } \varkappa(\mathfrak{G})=\varkappa(K).
\end{equation}

\end{theorem}

\begin{proof}
A sketch of the proof is indicated in
\cite{My}
and
\cite[\S\ 2.3, pp. 476--478]{Ma2},
but the precise proof has been given by Hasse in
\cite[\S\ 27, pp. 164--175]{Ha2}.
We have elaborated Hasse's proof on Wikipedia:\\
\verb+https://en.wikipedia.org/wiki/Principalization_(algebra)+
\end{proof}



\subsection{Metabelian \(p\)-groups of maximal class}
\label{ss:MaximalClass}

Generally, let \(G\) be a finite \(p\)-group, for an arbitrary prime number \(p\ge 2\).
Assume that \(G\) is of order \(\lvert G\rvert=p^n\), where \(n\) is the \textit{logarithmic order},
of \textit{nilpotency class} \(c=\mathrm{cl}(G)=m-1\), where \(m\) is the \textit{index of nilpotency},
and of \textit{coclass} \(r=\mathrm{cc}(G)=n-c=n-m+1\).
Denote the descending \textit{lower central series} of \(G\) by

\begin{equation}
\label{eqn:LowerCentral}
G=\gamma_1{G}>\gamma_2{G}>\gamma_3{G}>\ldots>\gamma_{m-1}{G}>\gamma_m{G}=1,
\end{equation}

\noindent
where \(\gamma_j{G}=\lbrack\gamma_{j-1}{G},G\rbrack\), for \(j\ge 2\),
and the ascending \textit{upper central series} of \(G\) by

\begin{equation}
\label{eqn:UpperCentral}
1=\zeta_0{G}<\zeta_1{G}<\zeta_2{G}<\ldots<\zeta_{c-1}{G}<\zeta_c{G}=G,
\end{equation}

\noindent
where \(\zeta_j{G}/\zeta_{j-1}{G}=\mathrm{Centre}(G/\zeta_{j-1}{G})\), for \(j\ge 1\).



Now let \(G\) be a metabelian \(p\)-group of coclass \(r=\mathrm{cc}(G)=1\), whence \(n=m\).
Then the abelianization \(G/G^\prime\) is of type \(\lbrack p,p\rbrack\)
and \(G=\langle x,y\rangle\) can be generated by two elements \(x,y\).
Suppose that \(m\ge 3\), i.e., \(G\) is non-abelian.
Denote the \textit{two-step centralizer},
that is the centralizer of the two-step factor group
\(\gamma_2{G}/\gamma_4{G}\)
of the lower central series, by

\begin{equation}
\label{eqn:TwoStepCentralizer}
\begin{aligned}
\chi_2{G}             &:= \left\lbrace g\in G\mid\lbrack g,u\rbrack\in\gamma_4{G} \quad \text{ for all } u\in\gamma_2{G}\right\rbrace, \text{ that is,}\\
\chi_2{G}/\gamma_4{G} &= \mathrm{Centralizer}_{G/\gamma_4{G}}(\gamma_2{G}/\gamma_4{G}).
\end{aligned}
\end{equation}

\noindent
\(\chi_2{G}\) is the biggest subgroup of \(G\) such that
\(\lbrack\chi_2{G},\gamma_2{G}\rbrack\le\gamma_4{G}\).
It is characteristic, contains the commutator subgroup \(\gamma_2{G}\), and
coincides with \(G\) if and only if \(m=3\),
whence we have \(\gamma_2{G}<\chi_2{G}<G\) \(\Longleftrightarrow\) \(m\ge 4\).
Select \textit{normalized generators} of \(G=\langle x,y\rangle\) such that

\begin{equation}
\label{eqn:Generators}
x\in G\setminus\chi_2{G}, \text{ if } m\ge 4, \quad y\in\chi_2{G}\setminus\gamma_2{G}.
\end{equation}

\noindent
Starting with the main commutator \(s_2:=\lbrack y,x\rbrack\in\gamma_2{G}\)
define the \textit{higher iterated commutators} recursively by
\(s_j:=\lbrack s_{j-1},x\rbrack\in\gamma_j{G}\), for \(j\ge 3\).


The general theory of \(p\)-groups of maximal class can be found in Huppert
\cite[\S\ 14, p. 361]{Hp}
and Berkovich
\cite[\S\ 9, p. 114]{Bv}.

In the isomorphism class of a metabelian \(p\)-group \(G\) of maximal class,
and of order \(\lvert G\rvert=p^m\),
there exists a representative  \(G_a^{m}(z,w)\)
whose normalized generators satisfy the following relations
with a fixed system of parameters
\(a=(a(m-k),\ldots,a(m-1))\), \(w\), and \(z\),
according to Blackburn
\cite{Bl}
and Miech
\cite[p. 332]{Mi1}:

\begin{itemize}

\item
parametrized \textit{commutator relations} with parameters \(0\le a(m-\ell)\le p-1\) for \(1\le\ell\le k\),
in particular, with non-vanishing parameter \(a(m-k)>0\),

\begin{equation}
\label{eqn:CommutatorRel}
\lbrack s_j,y\rbrack=
\begin{cases}
\prod_{\ell=1}^{k-j+2} s_{m-\ell}^{a(m-j+2-\ell)}\in\lbrack\gamma_j{G},\chi_2{G}\rbrack=\gamma_{m-k+j-2}{G}, \text{ for } 2 \le j\le k+1,\\
1, \text{ for } j\ge k+2,
\end{cases}
\end{equation}

\item
parametrized \textit{relations for} \(p\)th \textit{powers} of the generators \(x\), \(y\) and of higher commutators \(s_j\),
with parameters \(0\le w,z\le p-1\), and \(v:=a(m-k)\) if \(k=p-2\), but \(v:=0\) if \(k\le p-3\),

\begin{equation}
\label{eqn:PowerRel}
\begin{aligned}
(xy^{j})^p                                              &= s_{m-1}^{w+zj+vj^2}\in\zeta_1{G}, \quad \text{ for } 0\le j\le p-1 \\
y^p\prod_{\ell=2}^p\,s_\ell^{\binom{p}{\ell}}           &= s_{m-1}^z\in\zeta_1{G}, \\
s_{j+1}^p\prod_{\ell=2}^p\,s_{j+\ell}^{\binom{p}{\ell}} &= 1, \quad \text{ for } 1\le j\le m-2,
\end{aligned}
\end{equation}

\item
parametrized \textit{nilpotency relations} with parameter \(m\ge 3\) (the index of nilpotency),

\begin{equation}
\label{eqn:NilpotencyRel}
s_j\ne 1, \text{ for } j\le m-1, \quad \text{ but } s_j=1, \text{ for } j\ge m,
\end{equation}

\noindent
expressing the polycyclic structure
\(\gamma_j{G}=\langle s_j,\gamma_{j+1}{G}\rangle\),
\(\gamma_j{G}/\gamma_{j+1}{G}\simeq C_p\), for \(j\ge 2\) and \(j\le m-1\),
with \textit{cyclic factors} (CF) of the lower central series,
which coincides here with the reverse upper central series, \(\zeta_j{G}=\gamma_{m-j}{G}\), for \(0\le j\le m-1\),

\item
and trivial \textit{metabelian relations} within \(G^\prime\),

\begin{equation}
\label{eqn:MetabelianRel}
\lbrack s_i,s_j\rbrack=1, \text{ for all } i,j\ge 2.
\end{equation}

\end{itemize}


\noindent
In the case of an index of nilpotency \(m\ge 4\), the commutator relation for \(s_2\) in Formula
\eqref{eqn:CommutatorRel}
explicitly describes the properties of the two-step centralizer
\(\chi_2{G}=\langle y,\gamma_2{G}\rangle\),
which are implicitly postulated by Formula
\eqref{eqn:TwoStepCentralizer}:

\begin{equation}
\label{eqn:Defect}
\lbrack\gamma_2{G},\chi_2{G}\rbrack=\gamma_{m-k}{G}=\zeta_k{G},
\end{equation}

\noindent
where \(k=k(G)\) denotes the \textit{defect of commutativity} of \(G\),
which displays the following variety of possible values:
\(k=0\) for \(3\le m\le 4\), \(0\le k\le m-4\) for \(m\ge 5\),
and \(0\le k\le\min(m-4,p-2)\) for \(m\ge p+1\),
according to Miech
\cite[p. 331]{Mi1}.


The maximal normal subgroups \(M_i\) of \(G\)
contain the commutator subgroup \(\gamma_2{G}\) of \(G\)
as a normal subgroup of index \(p\) and thus
are of the shape \(M_i=\langle g_i,\gamma_2{G}\rangle\).
We define a standard ordering by
\(g_1=y\) and \(g_i=xy^{i-2}\) for \(2\le i\le p+1\).
In particular,
\(M_1=\langle y,\gamma_2{G}\rangle=\chi_2{G}\)
is the two-step centralizer. In summary:

\begin{equation}
\label{eqn:MaximalSubgroups}
M_1=\langle y,\gamma_2{G}\rangle=\chi_2{G}, \quad \text{ and } \quad M_i=\langle xy^{i-2},\gamma_2{G}\rangle \quad \text{ for } 2\le i\le p+1.
\end{equation}



\subsection{Polarization of the transfer target type}
\label{ss:Polarization}

\noindent
The transfer target type \(\tau(G)\) of a non-abelian \(p\)-group \(G\)
whose second derived quotient \(G/G^{\prime\prime}\) is of maximal class
has a particularly simple form with
a single \textit{polarized} component
and \(p\) fixed components of type \(\lbrack p,p\rbrack\).
The polarization is due to the \textit{two-step centralizer} \(\chi_2{\mathfrak{G}}\)
of the metabelianization \(\mathfrak{G}:=G/G^{\prime\prime}\) of \(G\)
\cite[\S\ 3.1.1, p. 412]{Ma4}.

\begin{definition}
\label{dfn:NearlyHomocyclic}
Let \(p\) be a prime number and \(e\ge p-1\) be an integer.
By the \textit{nearly homocyclic} abelian \(p\)-group of order \(p^e\) we understand the abelian group

\begin{equation}
\label{eqn:NearlyHomocyclic}
A(p,e):=\left(\overbrace{p^{q+1},\ldots,p^{q+1}}^{r \text{ times}},\overbrace{p^q,\ldots,p^q}^{p-1-r \text{ times}}\right),
\end{equation}

\noindent
with \(r\) invariants \(p^{q+1}\) and \(p-1-r\) invariants \(p^q\),
where \(q\ge 1\) and \(0\le r<p-1\) denote the quotient and remainder
of the Euclidean division \(e=q\cdot (p-1)+r\) of \(e\) by \(p-1\).
Additionally, let

\begin{equation}
\label{eqn:Elementary}
A(p,r):=\left(\overbrace{p,\ldots,p}^{r \text{ times}}\right),
\end{equation}

\noindent
be the \textit{elementary} abelian \(p\)-group of rank \(r\) for \(1\le r<p-1\),
and \(A(p,0):=1\) the trivial group.

\end{definition}

The abelian group \(A(p,e)\) is homocyclic if and only if
\(e\) is a multiple of \(p-1\) or \(1\le e<p-1\).


\begin{theorem}
\label{thm:Polarization}
Let \(p\) be a prime number and \(G\) be a \(p\)-group
whose second derived quotient \(\mathfrak{G}:=G/G^{\prime\prime}\)
is of coclass \(\mathrm{cc}(\mathfrak{G})=1\), nilpotency class \(c:=\mathrm{cl}(\mathfrak{G})\ge 2\), and defect \(k:=k(\mathfrak{G})\),
but is neither isomorphic to the extra special \(p\)-group \(G_0^3(0,1)\) of order \(p^3\) and exponent \(p^2\)
nor to the Sylow \(p\)-subgroup \(G_0^{p+1}(1,0)\simeq\mathrm{Syl}_p\mathrm{A}_{p^2}\) of the alternating group of degree \(p^2\),
with odd \(p\ge 3\).
Then the transfer target type of \(G\) is given by

\begin{equation}
\label{eqn:Polarization}
\tau(G)=\left(A(p,c-k),\overbrace{\lbrack p,p\rbrack,\ldots,\lbrack p,p\rbrack}^{p \text{ times}}\right).
\end{equation}

\end{theorem}

\begin{proof}
According to Theorem
\ref{thm:RstrAPofCompleteCover},
\(G\) and \(\mathfrak{G}\) share a common restricted Artin pattern,
\(\mathrm{AP}(G)=\mathrm{AP}(\mathfrak{G})\).
In particular they share the same transfer target type
\(\tau(G)=\tau(\mathfrak{G})\).
Thus we have reduced the claim to the metabelian \(p\)-group \(\mathfrak{G}\) of maximal class,
as discussed in \S\
\ref{ss:MaximalClass}.

We use the standard ordering of the maximal subgroups \(M_i\unlhd\mathfrak{G}\) in Formula
\eqref{eqn:MaximalSubgroups}
and recall from
\cite[Cor. 3.1, p. 476]{Ma1}
that the commutator subgroups \(M_i^\prime\) are given by

\begin{equation}
\label{eqn:DerivedSubgroups}
M_1^\prime=\gamma_{m-k}\mathfrak{G}, \quad \text{ and }\quad M_i^\prime=\gamma_3{\mathfrak{G}} \text{ for } 2\le i\le p+1.
\end{equation}

\noindent
We start by showing that \(M_i/M_i^\prime\simeq\lbrack p,p\rbrack\) for \(2\le i\le p+1\).
To this end, we observe that
\(M_i=\langle xy^{i-2},\gamma_2{\mathfrak{G}}\rangle\) and \(\gamma_2{\mathfrak{G}}=\langle s_2,\gamma_3{\mathfrak{G}}\rangle\),
and thus \(M_i/M_i^\prime=\langle xy^{i-2},s_2,\gamma_3{\mathfrak{G}}\rangle/\gamma_3{\mathfrak{G}}\)
where \(x,y\) are the normalized generators in Formula
\eqref{eqn:Generators}
and \(s_2=\lbrack y,x\rbrack\) is the main commutator of \(\mathfrak{G}\).
We have to determine the order of the elements \(xy^{i-2}\), with \(2\le i\le p+1\), and \(s_2\) modulo \(\gamma_3{\mathfrak{G}}\).
The third equation in Formula
\eqref{eqn:PowerRel}
yields \(s_{2}^p\cdot\prod_{\ell=2}^p\,s_{1+\ell}^{\binom{p}{\ell}}=1\) for \(j=1\),
and thus \(s_{2}^p\in\gamma_3{\mathfrak{G}}\), i.e., \(\mathrm{ord}(s_{2})=p\).
The first equation in the same formula gives
\((xy^{j})^p=s_{m-1}^{w+zj+vj^2}\), for \(0\le j\le p-1\),
which is an element of \(\gamma_3{\mathfrak{G}}\) when \(m-1\ge 3\), that is \(m\ge 4\).
For \(m=3\), however, we certainly have \(k=0\) and thus \(v=0\),
but we must distinguish between the two extra special groups
\(G_0^3(0,0)\), where \(z=w=0\) and \((xy^{j})^p=1\),
and \(G_0^3(0,1)\), where \(z=0\), \(w=1\) and \((xy^{j})^p=s_2\).
In summary, we have \(\mathrm{ord}(xy^{i-2})=p^2\) and \(M_i/M_i^\prime\simeq\lbrack p^2\rbrack\) for \(G_0^3(0,1)\),
but \(\mathrm{ord}(xy^{i-2})=p\) otherwise,
which proves our claim.

We continue by showing that \(M_1/M_1^\prime\simeq A(p,c-k)\), where \(c=m-1\).
In this case, we have
\(M_1=\langle y,\gamma_2{\mathfrak{G}}\rangle\), \(\gamma_2{\mathfrak{G}}=\langle s_2,\ldots,s_{m-k-1},\gamma_{m-k}{\mathfrak{G}}\rangle\),
and thus \(M_1/M_1^\prime=\langle y,s_2,\ldots,s_{m-k-1},\gamma_{m-k}{\mathfrak{G}}\rangle/\gamma_{m-k}{\mathfrak{G}}\).
We put \(N:=m-k\) and use the third equation in Formula
\eqref{eqn:PowerRel},
that is, 
\[s_{j+1}^p\cdot s_{j+2}^{\binom{p}{2}}\cdots s_{j+p-1}^{\binom{p}{p-1}}\cdot s_{j+p}=1, \quad \text{ for } 1\le j\le m-2.\]
These power relations enable us
to compute the order of the elements \(y\) and \(s_j\), with \(2\le j\le N-1\), modulo \(\gamma_{N}{\mathfrak{G}}\)
by a finite nested double induction with respect to the quotient \(q\ge 0\) and the remainder \(0\le r<p-1\)
of the Euclidean division \(N-1=q\cdot (p-1)+r\).
Thereby we always have to observe that \(s_j\in\gamma_{N}{\mathfrak{G}}\), i.e. \(\mathrm{ord}(s_j)=1\), for \(j\ge N\),
and that the binomial coefficients \(\binom{p}{j}\) are multiples of \(p\) for \(1\le j\le p-1\).

The first inner induction on the remainders successively yields\\
\(s_{N-1}^p=1\), \(s_{N-2}^ps_{N-1}^{\binom{p}{2}}=1\), and so on until \(s_{N-p+1}^ps_{N-p+2}^{\binom{p}{2}}\cdots s_{N-1}^{\binom{p}{p-1}}=1\),\\
that is, we can determine the orders \(\mathrm{ord}(s_j)=p\), for \(N-(p-1)\le j\le N-1\), recursively.

The second inner induction on the remainders successively yields\\
\(s_{N-p}^ps_{N-1}=1\), \(s_{N-p-1}^ps_{N-p}^{\binom{p}{2}}s_{N-2}=1\), and so on until \(s_{N-2p+2}^ps_{N-2p+3}^{\binom{p}{2}}\cdots s_{N-p}^{\binom{p}{p-1}}s_{N-p+1}=1\),\\
that is, we can determine the orders \(\mathrm{ord}(s_j)=p^2\), for \(N-2(p-1)\le j\le N-(p-1)-1\), recursively,
but also expressions for \(s_j\) with \(N-(p-1)\le j\le N-1\) in terms of \(s_\ell\) with \(\ell<j\).

In the same manner, we continue to obtain the orders\\
\(\mathrm{ord}(s_j)=p^3\), for \(N-3(p-1)\le j\le N-2(p-1)-1\),
and so on until

\begin{equation}
\label{eqn:LastButOneRow}
\mathrm{ord}(s_j)=p^q, \quad \text{ for } N-q(p-1)\le j\le N-(q-1)(p-1)-1,
\end{equation}

\noindent
by outer induction on the quotients.

The last inner induction on the remainders will in general be shorter than the previous inner inductions and yields

\begin{equation}
\label{eqn:LastRow}
\mathrm{ord}(s_j)=p^{q+1}, \quad \text{ for } N-q(p-1)-r\le j\le N-q(p-1)-1,
\end{equation}

\noindent
but the recursive process terminates when we reach \(N-q(p-1)-(r-1)=2\), that is,\\
\(s_{2}^p\cdot s_{3}^{\binom{p}{2}}\cdots s_{p}^{\binom{p}{p-1}}\cdot s_{p+1}=1\),\\
which shows that the higher commutators \(s_{p+1},\ldots,s_{N-1}\) can be expressed in terms of \(s_2,\ldots,s_p\).

Finally we employ the second equation in Formula
\eqref{eqn:PowerRel},
that is, 
\[y^p\cdot s_2^{\binom{p}{2}}\cdots s_{p-1}^{\binom{p}{p-1}}\cdot s_p=s_{m-1}^z,\]
which formally corresponds to \(N-q(p-1)-r=1\) when we put \(s_1:=y\),
but now the right side is \(s_{m-1}^z\) instead of \(1\).
Here we must be careful and distinguish the exceptional case of
the Sylow \(p\)-subgroup \(G_0^{p+1}(1,0)\simeq\mathrm{Syl}_p\mathrm{A}_{p^2}\) of the alternating group of degree \(p^2\),
if \(p\ge 3\) is an odd prime.
In this special case, we have \(m=p+1\), \(z=1\), and thus \(s_{m-1}^z=s_{p}\),
which has the unique effect that \(s_p\) cancels and cannot be expressed in terms of \(y,s_2,\ldots,s_{p-1}\).
All elements \(y,s_2,\ldots,s_{p-1},s_p\) are of order \(p\) and generate an elementary abelian \(p\)-group
of rank \(p\), bigger than the rank \(p-1\) of all nearly homocyclic abelian \(p\)-groups.

In all other cases, the higher commutators \(s_p,s_{p+1},\ldots,s_{N-1}\) can be expressed in terms of \(y,s_2,\ldots,s_{p-1}\),
and the Formulas
\eqref{eqn:LastButOneRow}
and
\eqref{eqn:LastRow}
show that \(M_1/M_1^\prime\) is isomorphic to the nearly homocyclic abelian \(p\)-group \(A(p,c-k)\)
of logarithmic order \(N-1=m-k-1=c-k\).
\end{proof}

Note that Theorem
\ref{thm:Polarization}
is more general than Blackburn's Theorem
\cite[Thm. 3.4, p. 68]{Bl},
since we have used Theorem
\ref{thm:RstrAPofCompleteCover}
and therefore the \(p\)-group \(G\) need not be of maximal class.



\section{Monotony of Artin patterns on descendant trees}
\label{s:MonotonyDescendantTrees}

\begin{definition}
\label{dfn:DescTree}
Let \(p\) be a prime 
and \(G\), \(H\) and \(R\) be finite \(p\)-groups.

\begin{enumerate}

\item
The \textit{lower central series} (LCS) of \(G\) is defined recursively by

\begin{equation}
\label{eqn:LowerCentralSeries}
\gamma_1{G}:=G, \quad \text{ and } \quad \gamma_n{G}:=\left\lbrack\gamma_{n-1}{G},G\right\rbrack \text{ for } n\ge 2.
\end{equation}

\item
We call \(G\) an \textit{immediate descendant} (or \textit{child}) of \(H\),
and \(H\) \textit{the parent} of \(G\),
if \(H\simeq G/\gamma_c{G}\) is isomorphic to the image of
the natural projection \(\pi:\,G\to G/\gamma_c{G}\) of \(G\) onto the quotient by
the last non-trivial term \(\gamma_c{G}>1\) of the LCS of \(G\),
where \(c:=\mathrm{cl}(G)\) denotes the \textit{nilpotency class} of \(G\).
In this case, we consider the projection \(\pi\)
as a directed edge \(G\to H\) from \(G\) to \(H\simeq\pi{G}\),
and we speak about the \textit{parent operator} \(\pi\):

\begin{equation}
\label{eqn:ChildParent}
\pi:\,G\to \pi{G}=G/\gamma_c{G}\simeq H \quad \text{ with } \quad c=\mathrm{cl}(G).
\end{equation}

\item
We call \(G\) a \textit{descendant} of \(H\),
and \(H\) an \textit{ancestor} of \(G\),
if there exists a finite \textit{path} of directed edges

\begin{equation}
\label{eqn:DescendantAncestor}
\left(Q_j\to Q_{j+1}\right)_{0\le j<\ell} \quad \text{ such that } \quad G=Q_0 \text{ and } H=Q_\ell,
\end{equation}

\noindent
where \(\ell\ge 0\) denotes the \textit{path length}.

\item
The \textit{descendant tree} of \(R\), denoted by \(\mathcal{T}(R)\),
is the rooted directed tree with root \(R\)
having the isomorphism classes of all descendants of \(R\) as its \textit{vertices}
and all (child, parent)-pairs \((G,H)\) among the descendants \(G,H\) of \(R\)
as its \textit{directed edges} \(G\to\pi{G}\simeq H\).
By means of formal iterations \(\pi^j\) of the parent operator \(\pi\),
each vertex of the descendant tree \(\mathcal{T}(R)\) can be connected with the root \(R\)
by a finite path of edges:

\begin{equation}
\label{eqn:IteratedParentPath}
\mathcal{T}(R)=\left\lbrace G\mid G=\pi^0{G}\to\pi^1{G}\to\pi^2{G}\to\ldots\to\pi^\ell{G}=R \quad \text{ for some } \quad \ell\ge 0\right\rbrace.
\end{equation}

\end{enumerate}

\end{definition}


\noindent
The restricted Artin pattern enjoys the following monotony property on a descendant tree.

\begin{theorem}
\label{thm:Monotony}
Let \(\mathcal{T}(R)\) be the descendant tree with root \(R>1\), a finite non-trivial \(p\)-group,
and let \(G\to\pi{G}\) be a directed edge of the tree.
Then the restricted Artin pattern \(\mathrm{AP}=(\tau,\varkappa)\)
satisfies the following monotonicity relations

\begin{equation}
\label{eqn:Monotony}
\begin{aligned}
     \tau(G) &\ge \tau(\pi{G}), \\
\varkappa(G) &\le \varkappa(\pi{G}),
\end{aligned}
\end{equation}

\noindent
that is, the TTT \(\tau\) is an isotonic mapping
and the TKT \(\varkappa\) is an antitonic mapping
with respect to the partial order \(G>\pi{G}\)
induced by the directed edges \(G\to\pi{G}\).

\end{theorem}

\begin{proof}
An abelian \(p\)-group \(A\) has the trivial group \(\pi{A}=A/\gamma_c{A}\simeq 1\) as its parent,
because the last non-trivial lower central \(\gamma_c{A}=\gamma_1{A}=A\) of \(A\) coincides with \(A\).
Since the tree root \(R>1\) is supposed to be non-trivial,
none of its proper descendants \(G>R\) can be abelian. \(G\) must be of class \(c=\mathrm{cl}(G)\ge 2\).
Consequently, if \(\left(G,\pi{G}\right)\) is a (child, parent)-pair of the tree \(\mathcal{T}(R)\),
then the last non-trivial lower central of the child \(G\) is contained in the commutator subgroup of \(G\),
i.e., \(\ker(\pi)=\gamma_c{G}\le\gamma_2{G}=G^\prime\),
as pointed out in
\cite[Thm. 5.3, p. 85]{Ma11}.

The restricted Artin pattern of a finite \(p\)-group \(G\)
is the pair  \(\mathrm{AP}(G)=\left(\tau(G),\varkappa(G)\right)\)
consisting of the transfer target type \(\tau(G)=\left(U/U^\prime\right)_{G^\prime<U<G}\)
and the transfer kernel type \(\varkappa(G)=\left(\ker T_{G,U}\right)_{G^\prime<U<G}\),
where \(T_{G,U}:G\to U/U^\prime\) denotes the Artin transfer
from \(G\) to the abelianization \(U/U^\prime\) of an intermediate group \(G^\prime<U<G\).
For each of these intermediate groups, we have \(\ker(\pi)\le G^\prime<U\),
which admits several statements about the epimorphism  \(\pi:\,G\to\pi{G}\).

According to
\cite[Prop.5.1, p. 82]{Ma11},
the mapping \(\pi\) is a bijection between the following systems of subgroups,
\(\mathcal{U}:=\lbrace U\mid G^\prime<U<G\rbrace\) and
\(\mathcal{V}:=\lbrace V\mid(\pi{G})^\prime=\pi(G^\prime)<V=\pi{U}<\pi{G}\rbrace\),
which is a necessary condition for the uniform comparability of the Artin patterns
\(\mathrm{AP}(G)\) and \(\mathrm{AP}(\pi{G})\)
of (child, parent)-pairs \((G,\pi{G})\) on the tree \(\mathcal{T}(R)\).

The TTT \(\tau\) is an isotonic mapping on the tree \(\mathcal{T}(R)\),
since \(\pi{U}/\pi{U}^\prime\) is an epimorphic image of \(U/U^\prime\),
according to
\cite[Thm. 5.1, p. 78]{Ma11},
and this property was used to define a partial order \(\pi{U}/\pi{U}^\prime\le U/U^\prime\)
on the components of the TTT, in
\cite[Dfn. 5.1, p. 80]{Ma11}.
Combining all components of the TTT, we obtain \(\tau(\pi{G})\le\tau(G)\), by
\cite[Dfn. 5.4, pp. 83--84]{Ma11}.
The (non-strict) inequality has the same direction as \(\pi{G}<G\).

The TKT \(\varkappa\) is an antitonic mapping on the tree \(\mathcal{T}(R)\),
since \(\pi(\ker(T_{G,U}))\le\ker(T_{\pi{G},\pi{U}})\),
according to
\cite[Thm. 5.2, p. 80]{Ma11},
and this property was used to define a partial order \(\ker(T_{G,U})\le\ker(T_{\pi{G},\pi{U}})\)
on the components of the TKT, in
\cite[Dfn. 5.2, p. 82]{Ma11}.
Combining all components of the TKT, we obtain \(\varkappa(G)\le\varkappa(\pi{G})\), by
\cite[Dfn. 5.4, pp. 83--84]{Ma11}.
The (non-strict) inequality has the opposite direction as \(G>\pi{G}\).
\end{proof}

\begin{remark}
\label{rmk:Monotony}
Due to the transitivity of all partial orders involved, Theorem
\ref{thm:Monotony}
remains true, when the (child, parent)-pair \((G,\pi(G))\)
is replaced by any (descendant, ancestor)-pair \((G,H)\).
\end{remark}


For the proof of our principal results,
we require some additional inheritance properties of group theoretic invariants
with respect to the (child, parent)-relation of a descendant tree.
The inheritance is directed from the child to its parent.

\begin{definition}
\label{dfn:SigmaGroup}
A finite or infinite pro-\(p\) group \(G\), with an odd prime \(p\ge 3\),
is called a \(\sigma\)-\textit{group},
if it possesses a \textit{generator inverting} (GI-)automorphism \(\sigma\in\mathrm{Aut}(G)\)
which acts as the inversion mapping on the derived quotient \(G/G^\prime\),

\begin{equation}
\label{eqn:SigmaGroup}
\sigma(g)G^\prime=g^{-1}G^\prime \quad \text{ for all } g\in G.
\end{equation}

\end{definition}


\begin{proposition}
\label{prp:Inheritance}
Let \(\pi:\,G\to\pi{G}\) be an epimorphism of groups
and let \(p\ge 2\) be a prime number.

\begin{enumerate}

\item
Let \(G\) be a pro-\(p\) group, with an odd prime \(p\ge 3\).\\
If \(G\) is a \(\sigma\)-group, then \(\pi{G}\) is also a \(\sigma\)-group.

\item
Let \(G\) be a finite \(p\)-group
and let \(\pi:\,G\to\pi{G}\) be the parent operator \(\pi{G}=G/\gamma_c{G}\) with \(c=\mathrm{cl}(G)\).

\begin{itemize}
\item
The nilpotency class increases exactly by \(1\) from the parent \(\pi{G}\) to the child \(G\).
\item
If \(s\) denotes the step size from parent \(\pi{G}\) to child \(G\),
then the coclass of \(G\) is given by \(\mathrm{cc}(G)=\mathrm{cc}(\pi{G})+(s-1)\),
in particular, \(\mathrm{cc}(G)=\mathrm{cc}(\pi{G})\) for \(s=1\).
\item
The derived length of \(\pi{G}\) is not bigger than the derived length of \(G\).
\end{itemize}

\begin{equation}
\label{eqn:Inheritance}
\mathrm{cl}(G)=\mathrm{cl}(\pi{G})+1, \quad \mathrm{cc}(G)=\mathrm{cc}(\pi{G})+(s-1), \quad  \quad \mathrm{dl}(G)\ge\mathrm{dl}(\pi{G}).
\end{equation}

\item
If \(G\) is a finite metabelian \(p\)-group of maximal class and \(\pi\) is the parent operator,
then the defect of the parent \(\pi{G}\) is not bigger than the defect of the child \(G\),
that is, \(k(\pi{G})\le k(G)\), more precisely

\begin{equation}
\label{eqn:InheritanceDefect}
k(\pi{G})=
\begin{cases}
k(G)   & \text{ if } k(G)=0, \\
k(G)-1 & \text{ if } k(G)\ge 1,
\end{cases}
\end{equation}

\noindent
which can also be expressed in the form \(k(\pi{G})=\max\left(0,k(G)-1\right)\).

\end{enumerate}

\end{proposition}

\begin{proof}
Let \(d:=\mathrm{dl}(G)\) be the derived length of \(G\) and
\[G=G^{(0)}>G^{(1)}>\ldots >G^{(d-1)}>G^{(d)}=1\]
be the derived series of \(G\), defined recursively by
\(G^{(0)}:=G\) and \(G^{(j)}:=\lbrack G^{(j-1)},G^{(j-1)}\rbrack\) for \(j\ge 1\).
For any homomorphism \(\pi:\,G\to H\) of groups, the proof of
\cite[Cor. 7.2, Appendix, p. 100]{Ma11}
has shown that
\((\pi{G})^{(j)}=\pi\left(G^{(j)}\right)\) for \(j\ge 0\),
and similarly for the lower central series
\(\gamma_j(\pi{G})=\pi\left(\gamma_j{G}\right)\) for \(j\ge 1\).

\begin{enumerate}

\item
The inheritance of a GI-automorphism has been proved in
\cite[Thm. 7.2, App., p. 102]{Ma11}.

\item
The lower central series of \(\pi{G}\) is given by\\
\(\pi{G}>\gamma_2\pi{G}=\pi\gamma_2{G}=\gamma_2{G}/\gamma_c{G}>\ldots>\gamma_{c-1}\pi{G}=\pi\gamma_{c-1}{G}=\gamma_{c-1}{G}/\gamma_c{G}>\gamma_{c}\pi{G}\simeq 1\),\\
whence \(\mathrm{cl}(\pi{G})=c-1\).

Let \(n\) be the logarithmic order of \(G\), that is \(\lvert G\rvert=p^n\).
If \(\lvert\gamma_c{G}\rvert=p^s\),
then we have \(\lvert \pi{G}\rvert=(G:\gamma_c{G})=p^{n-s}\) and \(\pi{G}\) is of logarithmic order \(n-s\).
Consequently, the coclass of \(\pi{G}\) is given by
\(\mathrm{cc}(\pi{G})=(n-s)-(c-1)=(n-c)-(s-1)=\mathrm{cc}(G)-(s-1)\).

The image of the last non-trivial term \(G^{(d-1)}\) of the derived series of \(G\) under \(\pi\)
is given by \(\pi{G}^{(d-1)}=G^{(d-1)}\ker{\pi}/\gamma_c{G}\), where \(\ker{\pi}=\gamma_c{G}\).
If we denote by \(d^\prime:=\mathrm{dl}(\pi{G})\) the derived length of \(\pi{G}\),
then the derived series of \(\pi{G}\) is given by
\[\pi{G}=\pi{G}^{(0)}>\pi{G}^{(1)}>\ldots >\pi{G}^{(d^\prime-1)}>\pi{G}^{(d^\prime)}=1,\]
where \(d^\prime<d\) if \(G^{(d-1)}\le\gamma_c{G}\), and \(d^\prime=d\) otherwise.

\item
If the order of \(G\) is given by \(\lvert G\rvert=p^m\),
then the two-step centralizer of \(G\) satisfies the condition
\(\lbrack\chi_2{G},\gamma_2{G}\rbrack=\gamma_{m-k}{G}\),
where \(k:=k(G)\le m-4\) denotes the defect of \(G\).
Application of the epimorphism \(\pi\) yields
\[\lbrack\chi_2\pi{G},\gamma_2\pi{G}\rbrack=\lbrack\pi\chi_2{G},\pi\gamma_2{G}\rbrack=\pi\lbrack\chi_2{G},\gamma_2{G}\rbrack=\pi\gamma_{m-k}{G}=\gamma_{m-k}\pi{G},\]
where
\[
\gamma_{m-k}\pi{G}=
\begin{cases}
1                          & \text{ if } k=0, \\
\gamma_{(m-1)-(k-1)}\pi{G} & \text{ if } k\ge 1.
\end{cases}
\]
This proves the claim, since the step size for coclass \(1\) is \(s=1\) and thus \(\lvert\pi{G}\rvert=p^{m-1}\).\\
Here, we have used that \(\pi\chi_2{G}=\chi_2\pi{G}\), since applying the surjection \(\pi\) to the equation
\[\chi_2{G}=\left\lbrace g\in G\mid\lbrack g,u\rbrack\in\gamma_4(G) \quad \text{ for all } u\in\gamma_2(G)\right\rbrace\]
and observing that \(\pi\gamma_2{G}=\gamma_2\pi{G}\) and \(\pi\gamma_4{G}=\gamma_4\pi{G}\) yields
\[\pi\chi_2{G}=\left\lbrace \pi{g}\in\pi{G}\mid\lbrack\pi{g},\pi{u}\rbrack\in\gamma_4\pi{G} \quad \text{ for all } \pi{u}\in\gamma_2\pi{G}\right\rbrace.\qedhere\]
\end{enumerate}
\end{proof}



\section{Proof of the main results}
\label{s:Proofs}

We are now in the position to prove Theorem
\ref{thm:MainTheorem}
and Corollary
\ref{cor:MainTheorem}
simultaneously.
The dominant part of the proof
is valid for \textit{any} number field \(K\)
and admits \textit{two possibilities}
\(\ell_5{K}\in\lbrace 2,3\rbrace\)
for the length of its \(5\)-class tower.
The final specialization to a real quadratic field
\(K=\mathbb{Q}(\sqrt{d})\)
with discriminant \(d>0\) will enforce the fixed length
\(\ell_5{K}=3\).

\begin{proof}
Let \(K\) be an \textit{arbitrary} algebraic number field
with \(5\)-class group
\(\mathrm{Cl}_5{K}\simeq\lbrack 5,5\rbrack\)
and \(5\)-capitulation type
\(\varkappa(K)=\left(\ker{j_{L_i\vert K}}\right)_{1\le i\le 6}\sim (1,0^5)\)
in its six unramified cyclic quintic extensions \(L_1,\ldots,L_6\),
and suppose that the \(5\)-class groups of the latter are given by
\(\tau(K)=\left(\mathrm{Cl}_5{L_i}\right)_{1\le i\le 6}\sim\left(\lbrack 25,5,5,5\rbrack,\lbrack 5,5\rbrack^5\right)\).

Then the restricted Artin pattern of \(K\) is
\(\mathrm{AP}(K)=\left(\tau(K),\varkappa(K)\right)\)
and, by Theorem
\ref{thm:ArtinPattern},
Formula
\eqref{eqn:ArtinPattern},
which constitutes the translation from number theory to group theory,
\(\mathrm{AP}(K)\) coincides with the restricted Artin pattern
\(\mathrm{AP}(\mathfrak{G})=\left(\tau(\mathfrak{G}),\varkappa(\mathfrak{G})\right)\)
of the second \(5\)-class group \(\mathfrak{G}=\mathrm{G}_5^2{K}\),
where
\(\tau(\mathfrak{G})=\left(M_i/M_i^\prime\right)_{1\le i\le 6}\)
and
\(\varkappa(\mathfrak{G})=\left(\ker{T_{\mathfrak{G},M_i}}\right)_{1\le i\le 6}\)
in terms of the six maximal subgroups \(M_1,\ldots,M_6\) of \(\mathfrak{G}\)
and the Artin transfers \(T_{\mathfrak{G},M_i}\).

According to Formula
\eqref{eqn:GeneratorRank},
\(\mathfrak{G}\) possesses the abelianization
\(\mathfrak{G}/\mathfrak{G}^\prime\simeq\mathrm{Cl}_5{K}\simeq\lbrack 5,5\rbrack\).
Hence, the metabelian \(5\)-group \(\mathfrak{G}\) must be a descendant of the abelian \(5\)-group
\(\langle 5^2,2\rangle\simeq\lbrack 5,5\rbrack\),
and for identifying \(\mathfrak{G}\) by the strategy of \textit{pattern recognition via Artin transfers}
we therefore have to start the \(p\)-group generation algorithm
\cite{HEO}
at the root
\(R=\langle 5^2,2\rangle\)
and to construct the descendant tree \(\mathcal{T}(R)\),
always looking for the \textit{assigned search pattern} \(\mathrm{AP}(\mathfrak{G})\).

The recursive steps of this algorithm by Newman
\cite{Nm1}
and O'Brien
\cite{Ob}
may be monitored
on the metabelian skeleton of the coclass-\(1\) subtree \(\mathcal{T}^1(R)\) in
\cite[Fig. 3.3, p. 425]{Ma4},
on the interface between coclass \(1\) and \(2\) in
\cite[Fig. 3.8, p. 448]{Ma4},
and on the pruned descendant tree containing the decisive bifurcation in Figure
\ref{fig:PrunedTree5x5}.

In each step,
a parent \(\pi(G)\) gives rise to several children \(G\),
and the TTT
\(\tau(G)\ge\tau(\pi(G))\)
is isotonic with respect to the partial order
\(G>\pi(G)\)
induced by the tree edges
\(G\to\pi(G)\),
wheras the TKT
\(\varkappa(G)\le\varkappa(\pi(G))\)
is antitonic,
by the \textit{fundamental monotony result} in Theorem
\ref{thm:Monotony}.

The \textit{pruning} process will run parallel to the recursive tree construction
in the following way.
A vertex with a first TTT component
\(\tau_1>\lbrack 25,5,5,5\rbrack\)
or with a higher TTT component
\(\tau_i>\lbrack 5,5\rbrack\), for some \(2\le i\le 6\),
must be discarded together with all its descendants,
since their TTT \(\tau\) will contain an inadequate component.
Similarly, a vertex having any \textit{forbidden} TKT \(\varkappa\)
different from the only two \textit{admissible} TKTs,
\(\varkappa\sim (1,0^5)\) and the biggest possible \(\varkappa=(0^6)\),
must be cancelled together with the entire set of its descendants,
since their TKT \(\varkappa\) will also be forbidden.

An important \textit{additional filter} must be used to eliminate
the numerous metabelian \(5\)-groups of maximal class with positive defect of commutativity, \(k\ge 1\),
although they share the common admissible TKT \(\varkappa=(0^6)\), by
\cite[Thm. 2.5.(3), p. 479]{Ma2}.
However, all descendants of a parent with \(k\ge 1\) will also have \(k\ge 1\),
by item (3) of Proposition
\ref{prp:Inheritance},
Formula
\eqref{eqn:InheritanceDefect},
and a metabelian group of maximal class with \(k\ge 1\) reveals
a \textit{total stabilization} of all components of the restricted Artin pattern,
according to Theorem
\cite[Thm. 6.2.(2), p. 92]{Ma11}.
With the aid of Theorem
\ref{thm:RstrAPofCompleteCover},
Formula
\eqref{eqn:RstrAPofCompleteCover},
we conclude that non-metabelian \(5\)-groups \(G\) of any coclass,
whose metabelianization \(G/G^{\prime\prime}\) is of maximal class with \(k\ge 1\),
will also have the TKT \(\varkappa=(0^6)\) (and not the desired TKT \(\varkappa\sim (1,0^5)\))
and, due to \(k\ge 1\), the inadequate first TTT component which was determined in Theorem
\ref{thm:Polarization}.
(A metabelian \(p\)-group of maximal class with order \(p^m\) has defect \(k\ge 1\)
if and only if \(\lvert\tau_1\rvert<m-1\).)

After these preliminary considerations,
we execute the actual steps for finding \textit{successive approximations}
in form of \textit{class-\(c\) quotients} \(G/\gamma_{c+1}{G}\)
for the \(5\)-class tower group
\(G:=\mathrm{G}_5^\infty{K}\)
of the field \(K\).
The steps are summarized in Table
\ref{tbl:SuccessiveApproximation}.

\begin{itemize}

\item
The abelian root
\(R=\langle 5^2,2\rangle\simeq C_5\times C_5\)
is the class-\(1\) quotient \(G/\gamma_2{G}\) of \(G\),
as shown by Formula
\eqref{eqn:GeneratorRank}.
It has
\(\tau=\left(\lbrack 5\rbrack^6\right)\), \(\varkappa=(0^6)\),
and gives rise to the extra special \(5\)-groups as its \(2\) children, the capable
\(\langle 5^3,3\rangle\simeq G_0^3(0,0)\)
with
\(\tau=\left(\lbrack 5,5\rbrack^6\right)\), and admissible \(\varkappa=(0^6)\),
and the terminal
\(\langle 5^3,4\rangle\simeq G_0^3(0,1)\)
with
\(\tau\sim\left(\lbrack 5,5\rbrack,\lbrack 25\rbrack^5\right)\), and forbidden \(\varkappa=(1^6)\).\\
(For a real quadratic field \(K\), the group \(\langle 5^3,4\rangle\) cannot occur as \(\mathrm{G}_5^2{K}\)
for the additional reason that it is not a \(\sigma\)-group.)

\item
The compulsory next parent
\(\langle 5^3,3\rangle\simeq G_0^3(0,0)\)
is the class-\(2\) quotient \(G/\gamma_3{G}\) of \(G\)
and reveals a \textit{bifurcation}
\cite[\S\ 8, p. 168]{Ma6}
with
\(4\) metabelian children of step size \(1\)
\cite[Fig. 3.3, p. 425]{Ma4},
and \(12\) metabelian children of step size \(2\)
\cite[Fig. 3.8, p. 448]{Ma4}.

Among the former, we have
a capable vertex \(\langle 5^4,7\rangle\) with admissible \(\varkappa=(0^6)\),
a single leaf \(\langle 5^4,8\rangle\) with the desired \(\varkappa\sim (1,0^5)\)
but too small \(\tau\sim\left(\lbrack 5,5,5\rbrack,\lbrack 5,5\rbrack^5\right)\),
and two leaves \(\langle 5^4,9\vert 10\rangle\) with forbidden \(\varkappa\sim (2,0^5)\).

Among the latter,
which are of coclass \(2\),
the \(11\) groups \(\langle 5^5,4\ldots 14\rangle\) have forbidden TKTs,
and only the group \(\langle 5^5,3\rangle\) has admissible \(\varkappa=(0^6)\).
However, the last five components of the TTT \(\tau=\left(\lbrack 5,5,5\rbrack^6\right)\)
disqualify the group as a possible class-\(3\) quotient of \(G\).

\item
For the class-\(3\) quotient \(G/\gamma_4{G}\) of \(G\)
we have to take
\(\langle 5^4,7\rangle\simeq G_0^4(0,0)\)
which is the next parent.
It has \(9\) metabelian children, all of step size \(1\).
The capable vertex \(\langle 5^5,30\rangle\) has admissible \(\varkappa=(0^6)\).
The leaf \(\langle 5^5,31\rangle\) has the desired \(\varkappa\sim (1,0^5)\)
but too small \(\tau\sim\left(\lbrack 5,5,5,5\rbrack,\lbrack 5,5\rbrack^5\right)\).
The leaf \(\langle 5^5,32\rangle\) has forbidden \(\varkappa\sim (2,0^5)\), and
the \(6\) groups \(\langle 5^5,33\ldots 38\rangle\) have positive defect of commutativity \(k=1\).\\
(For a real quadratic field \(K\), the groups \(\langle 5^5,31\ldots 38\rangle\) cannot occur as \(\mathrm{G}_5^2{K}\)
for the additional reason that they are not \(\sigma\)-groups.)

\item
Now we arrive at the \textit{crucial bifurcation}
\cite[\S\ 8, p. 168]{Ma6}
of the mandatory next parent
\(\langle 5^5,30\rangle\simeq G_0^5(0,0)\)
which is the class-\(4\) quotient \(G/\gamma_5{G}\) of \(G\).
It has \(21\) children of step size \(1\) (\(13\) of them metabelian)
and \(115\) non-metabelian children of step size \(2\).

Among the former, we have
a capable vertex \(\langle 5^6,630\rangle\) with admissible \(\varkappa=(0^6)\),
\(4\) vertices \(\langle 5^6,631\ldots 634\rangle\) with forbidden \(\varkappa\sim (2,0^5)\),
a single group \(\langle 5^6,635\rangle\) with both,
the desired \(\varkappa\sim (1,0^5)\) and the desired \(\tau\sim\left(\lbrack 25,5,5,5\rbrack,\lbrack 5,5\rbrack^5\right)\),
and \(7\) groups \(\langle 5^6,636\ldots 642\rangle\) with \(k=1\).

Among the latter,
which are of coclass \(2\),
we have \(3\) capable vertices \(\langle 5^7,360\vert 372\vert 384\rangle\) with admissible \(\varkappa=(0^6)\),
\(5\) leaves \(\langle 5^7,361\vert 373\vert 374\vert 385\vert 386\rangle\) with both,
the desired \(\varkappa\sim (1,0^5)\) and the desired \(\tau\sim\left(\lbrack 25,5,5,5\rbrack,\lbrack 5,5\rbrack^5\right)\),
\(32\) vertices \(\langle 5^7,n\rangle\), where \(n\) is listed in Table
\ref{tbl:SuccessiveApproximation},
with forbidden \(\varkappa\sim (2,0^5)\),
and \(75\) vertices \(\langle 5^7,392\ldots 466\rangle\) with \(k\ge 1\).

\item
To make sure that \(\langle 5^6,635\rangle\simeq G_0^5(0,1)\)
is the unique possibility for the metabelianization \(G/G^{\prime\prime}\) of \(G\),
we have to check the next parent \(\langle 5^6,630\rangle\simeq G_0^5(0,0)\).
It has \(12\) children of step size \(1\) (\(9\) of them metabelian),
but either their \(\tau\sim\left(\lbrack 25,25,5,5\rbrack,\lbrack 5,5\rbrack^5\right)\) is too big or they have \(k=1\).
Thus, the break-off condition on the coclass-\(1\) tree is reached.

\item
It remains to prove that \(\langle 5^7,361\vert 373\vert 374\vert 385\vert 386\rangle\)
are the only \(5\) non-metabelian possibilities for \(G\).
For this purpose, we must check the roots \(\langle 5^7,360\vert 372\vert 384\rangle\)
of \(3\) non-metabelian coclass-\(2\) trees,
which possess \(14\), resp. \(20\), resp. \(20\), children of step size \(1\),
but either their \(\tau\sim\left(\lbrack 25,25,5,5\rbrack,\lbrack 5,5\rbrack^5\right)\) is too big or they have \(k\ge 1\).

\end{itemize}

\noindent
The proof for a real quadratic field \(K\) will be finished in Theorem
\ref{thm:ShafarevichCover}.\qedhere
\end{proof}

\renewcommand{\arraystretch}{1.1}

\begin{table}[ht]
\caption{Descendants of class-\(c\) quotients \(G/\gamma_{c+1}{G}\) of the \(5\)-class tower group \(G\)}
\label{tbl:SuccessiveApproximation}
\begin{center}
\begin{tabular}{|r|l|r|c||c|r|c|r|c|r|c|r|c|}
\hline
\multicolumn{4}{|c||}{\(G/\gamma_{c+1}{G}\)} & & \multicolumn{2}{|c|}{\(\varkappa=(0^6)\)} & \multicolumn{2}{|c|}{\(\varkappa\sim (1,0^5)\)} & \multicolumn{2}{|c|}{\(\varkappa\sim (2,0^5)\)} & \multicolumn{2}{|c|}{\(\varkappa=(0^6)\)} \\
\multicolumn{3}{|c|}{}    & \(\#\) of       & step   & \multicolumn{2}{|c|}{\(k=0\)} & \multicolumn{2}{|c|}{\(k=0\)} & \multicolumn{2}{|c|}{\(k=0\)} & \multicolumn{2}{|c|}{\(k=1\)}\\
 \(c\) & ord     &      id & descendants     & size   & \(\#\) &       id & \(\#\) &               id & \(\#\)  &                id & \(\#\)  &                  id \\
\hline
 \(1\) & \(5^2\) &   \(2\) & \(2/1\)         &  \(1\) &  \(1\) & \(3\)    &        &                  &         &                   &         &                     \\
\hline
 \(2\) & \(5^3\) &   \(3\) & \(4/1;12/6\)    &  \(1\) &  \(1\) & \(7\)    &  \(1\) & \(8\)            &   \(2\) & \(9\vert 10\)     &         &                     \\
       &         &         &                 &  \(2\) &  \(1\) & \(3\)    &        &                  &         &                   &         &                     \\
\hline
 \(3\) & \(5^4\) &   \(7\) & \(9/2\)         &  \(1\) &  \(1\) & \(30\)   &  \(1\) & \(31\)           &   \(1\) & \(32\)            &   \(6\) & \(33\ldots 38\)     \\
\hline
 \(4\) & \(5^5\) &  \(30\) & \(21/2;115/13\) &  \(1\) &  \(1\) & \(630\)  &  \(1\) & \(635\)          &   \(4\) & \(631\ldots 634\) &  \(15\) & \(636\ldots 650\)   \\
       &         &         &                 &  \(2\) &  \(3\) & \(360\)  &  \(5\) & \(361\)          &  \(32\) & \(352\ldots 359\) &  \(75\) & \(392\ldots 466\)   \\
       &         &         &                 &        &        & \(372\)  &        & \(373\vert 374\) &         & \(362\ldots 371\) &         &                     \\
       &         &         &                 &        &        & \(384\)  &        & \(385\vert 386\) &         & \(375\ldots 383\) &         &                     \\
       &         &         &                 &        &        &          &        &                  &         & \(387\ldots 391\) &         &                     \\
\hline
\hline
       & \(5^6\) & \(630\) & \(12/2\)        &  \(1\) &  \(1\) & \(1287\) &  \(1\) & \(1288\)         &   \(1\) & \(1286\)          &   \(9\) & \(1289\ldots 1297\) \\
\hline
       & \(5^7\) & \(360\) & \(14/2\)        &  \(1\) &  \(1\) & \(2\)    &  \(1\) & \(4\)            &   \(2\) & \(1\vert 3\)      &  \(10\) & \(5\ldots 14\)      \\
\hline
       & \(5^7\) & \(372\) & \(20/2\)        &  \(1\) &  \(1\) & \(2\)    &  \(1\) & \(5\)            &   \(3\) &\(1\vert 3\vert 4\)&  \(15\) & \(6\ldots 20\)      \\
\hline
       & \(5^7\) & \(384\) & \(20/2\)        &  \(1\) &  \(1\) & \(2\)    &  \(1\) & \(5\)            &   \(3\) &\(1\vert 3\vert 4\)&  \(15\) & \(6\ldots 20\)      \\
\hline
\end{tabular}
\end{center}
\end{table}

\begin{remark}
\label{rmk:Proofs}
It is very important to point out that, for the sake of homogeneity, we have conducted both,
the search for the metabelianization \(\mathfrak{G}=G/G^{\prime\prime}\) of the \(5\)-class tower group \(G\),
and the search for \(G\) itself,
within the frame of the \(p\)-group generation algorithm.
In two previous papers,
we have developed the tools for finding the second \(5\)-class group \(\mathfrak{G}\) much more directly:
The metabelian skeleton of the coclass tree \(\mathcal{T}^1(\langle 25,2\rangle)\) in
\cite[Fig. 3.3, p. 425]{Ma4}
is embedded within a sort of coordinate system
with the TKT \(\varkappa\) as its horizontal axis
and the polarized first TTT component \(\tau_1\) as its vertical axis.
According to
\cite[Thm. 2.5.(3), p. 479]{Ma2},
the TKT \(\varkappa(\mathfrak{G})\sim (1,0^5)\) uniquely determines the parameters \(w=1\), \(z=0\), and \(k=0\)
of the metabelian \(5\)-group \(\mathfrak{G}\simeq G_a^m(z,w)=G_0^m(0,1)\),
since \(k=0\) if and only if \(a\) is the empty family.
As second input data we do not even need the exact TTT \(\tau(\mathfrak{G})=\left(\lbrack 25,5,5,5\rbrack,\lbrack 5,5\rbrack^5\right)\).
The order of the polarized first TTT component \(\lvert\tau_1\rvert=25\cdot 5\cdot 5\cdot 5=5^5\) suffices, since according to
\cite[Thm. 3.1.(3), p. 475]{Ma1},
the relation \(\lvert\tau_1\rvert=\lvert M_1/M_1^\prime\rvert=5^{m-k-1}=5^{m-1}\) uniquely determines the index of nilpotency as \(m=6\).
Therefore, we have \(\mathfrak{G}\simeq G_0^6(0,1)\), and this is exactly the group \(\langle 5^6,635\rangle\).
\end{remark}



\section{Relation rank and generator rank}
\label{s:RelationRank}

\begin{theorem}
\label{thm:RelationRank}

(Shafarevich)
Let \(p\) be a prime number
and denote by \(\zeta\) a primitive \(p\)th root of unity.
Let \(K\) be a number field with signature \((r_1,r_2)\) and
torsionfree Dirichlet unit rank \(r=r_1+r_2-1\),
and let \(S\) be a finite set of non-archimedean or real archimedean places of \(K\).
Assume that no place in \(S\) divides \(p\).

Then the relation rank \(d_2{G_S}:=\dim_{\mathbb{F}_p}H^2(G_S,\mathbb{F}_p)\) of
the Galois group \(G_S:=\mathrm{Gal}(K_S\vert K)\) of
the maximal pro-\(p\) extension \(K_S\) of \(K\)
which is unramified outside of \(S\)
is bounded from above by

\begin{equation}
\label{eqn:RelationRank}
d_2{G_S}\le
\begin{cases}
d_1{G_S}+r   & \text{ if } S\ne\emptyset \text { or } \zeta\notin K, \\
d_1{G_S}+r+1 & \text{ if } S=\emptyset  \text { and } \zeta\in K,
\end{cases}
\end{equation}

\noindent
where \(d_1{G_S}:=\dim_{\mathbb{F}_p}H^1(G_S,\mathbb{F}_p)\) denotes the generator rank of \(G_S\).

\end{theorem}

\begin{proof}
The original statement in
\cite[Thm. 6, \((18^\prime)\)]{Sh}
contained a serious misprint
which was corrected in
\cite[Thm. 5.5, p. 28]{Ma10}.
\end{proof}



\section{Covers and Shafarevich covers}
\label{s:ShafarevichCover}

\noindent
We begin with a purely group theoretic concept,
which aids in getting an overview of all extensions of a particular kind
of an assigned \(p\)-group.

\begin{definition}
\label{dfn:Cover}
Let \(p\) be a prime and
\(\mathfrak{G}\) be a finite \textit{metabelian} \(p\)-group.
By the \textit{cover} of \(\mathfrak{G}\) we understand
the set of all (isomorphism classes of) finite \(p\)-groups
whose second derived quotient is isomorphic to \(\mathfrak{G}\):

\begin{equation}
\label{eqn:Cover}
\mathrm{cov}(\mathfrak{G}):=\left\lbrace G\mid \mathrm{ord}(G)<\infty,\ G/G^{\prime\prime}\simeq\mathfrak{G}\right\rbrace.
\end{equation}

\noindent
By eliminating the finiteness condition, we obtain the \textit{complete cover} of \(\mathfrak{G}\),

\begin{equation}
\label{eqn:CompleteCover}
\mathrm{cov}_c(\mathfrak{G}):=\left\lbrace G\mid G/G^{\prime\prime}\simeq\mathfrak{G}\right\rbrace.
\end{equation}

\end{definition}


\begin{remark}
\label{rmk:Cover}
The unique metabelian element of \(\mathrm{cov}(\mathfrak{G})\) is
the isomorphism class of \(\mathfrak{G}\) itself.
\end{remark}

\begin{theorem}
\label{thm:CoverMainResult}
The cover of the metabelian \(5\)-group \(\langle 15625,635\rangle\) consists of \(6\) elements,

\begin{equation}
\label{eqn:CoverMainResult}
\mathrm{cov}(\langle 15625,635\rangle)=
\biggl\lbrace\langle 15625,635\rangle\biggr\rbrace\bigcup\biggl\lbrace\langle 78125,n\rangle\mid n\in\lbrace 361, 373, 374, 385, 386\rbrace\biggr\rbrace.
\end{equation}

\end{theorem}

\begin{proof}
This is a consequence of the presentations in Formulas
\eqref{eqn:pcPresentations}
and
\eqref{eqn:RelatorWord}
together with
\S\
\ref{s:Proofs}.
\end{proof}


\begin{remark}
\label{rmk:CompleteCover}
We shall not make essential use of the complete cover in this article.
The only reason for its introduction is
the possibility to express certain assertions in a more general way.
\end{remark}

\begin{theorem}
\label{thm:RstrAPofCompleteCover}
Let \(\mathfrak{G}\) be a finite metabelian \(p\)-group.
Then all elements of the complete cover of \(\mathfrak{G}\) share a common restricted Artin pattern:

\begin{equation}
\label{eqn:RstrAPofCompleteCover}
 \mathrm{AP}(G)=\mathrm{AP}(\mathfrak{G}), \quad \text{ for all } \quad G\in\mathrm{cov}_c(\mathfrak{G}).
\end{equation}

\end{theorem}

\begin{proof}
This is the Main Theorem of
\cite[Thm. 5.4, p. 86]{Ma11}.
\end{proof}


Motivated by the Shafarevich Theorem
\ref{thm:RelationRank},
we connect the group theoretic cover concept with arithmetical information,
which will be useful for identifying the \(p\)-class tower group
of an assigned number field.

\begin{definition}
\label{dfn:ShafarevichCover}
Let \(p\) be a prime and
\(K\) be a number field with
\(p\)-class rank \(\varrho:=d_1(\mathrm{Cl}_p(K))\),
torsionfree Dirichlet unit rank \(r\),
and the second \(p\)-class group \(\mathfrak{G}:=\mathrm{G}_p^2 K\).
By the \textit{Shafarevich cover},
\(\mathrm{cov}(\mathfrak{G},K)\),
of \(\mathfrak{G}\) \textit{with respect to} \(K\)
we understand the subset of \(\mathrm{cov}(\mathfrak{G})\) whose elements \(G\)
satisfy the following condition for their relation rank:

\begin{equation}
\label{eqn:ShafarevichCover}
\varrho\le d_2{G}\le\varrho+r+\vartheta,
\quad \text{ where } \quad
\vartheta:=
\begin{cases}
1 & \text{ if } K \text{ contains the } p\text{th roots of unity,} \\
0 & \text{ otherwise.} 
\end{cases}
\end{equation}

\end{definition}


\begin{theorem}
\label{thm:ShafarevichCover}
The Shafarevich cover of the metabelian \(5\)-group \(\langle 15625,635\rangle\)
with respect to a real quadratic field \(K\) of \(5\)-class rank \(\varrho=2\) consists of \(5\) elements,

\begin{equation}
\label{eqn:ShafarevichCoverMainResult}
\mathrm{cov}(\langle 15625,635\rangle,K)=
\biggl\lbrace\langle 78125,n\rangle\mid n\in\lbrace 361, 373, 374, 385, 386\rbrace\biggr\rbrace.
\end{equation}

\end{theorem}

\begin{proof}
A real quadratic field has torsionfree unit rank \(r=0+2-1=1\)
and does not contain the primitive fifth roots of unity, i.e., \(\vartheta=0\).
The relation rank of a \(p\)-group is equal to its \(p\)-multiplicator rank
\cite[(26), p. 178]{Ma6}.
The metabelian \(5\)-group \(\mathfrak{G}=\langle 15625,635\rangle\) has relation rank \(4\),
whereas the relation rank of the non-metabelian \(5\)-groups
\(G=\langle 78125,n\rangle\) with \(n\in\lbrace 361, 373, 374, 385, 386\rbrace\)
is \(3\).
Consequently, the former satisfies the conflicting condition
\(d_2{\mathfrak{G}}=4>\varrho+r+\vartheta=3\),
whereas the latter enjoy the required property
\(2=\varrho\le d_2{G}=3\le\varrho+r+\vartheta=3\)
for a \(5\)-class tower group of \(K\).
\end{proof}

\begin{remark}
\label{rmk:Realization}
The arithmetical invariants \(\tau(K)\) and \(\varkappa(K)\)
of the real quadratic fields \(K\) in Example
\ref{exm:Realization}
were computed with the aid of MAGMA, version V2.21-11
\cite{BCP,BCFS,MAGMA},
under the LINUX operating system
on a machine with two XEON \(8\)-core CPUs and \(256\,\)GB RAM.
The unramified cyclic quintic extensions were constructed
by means of the class field routines written by Fieker
\cite{Fi}
for MAGMA.
\end{remark}




\section{Presentation and normal lattice of the \(5\)-tower groups}
\label{s:NormalLattice}

\begin{figure}[hb]
\caption{Pruned descendant tree of the abelian root \(\langle 25,2\rangle\)}
\label{fig:PrunedTree5x5}

\input{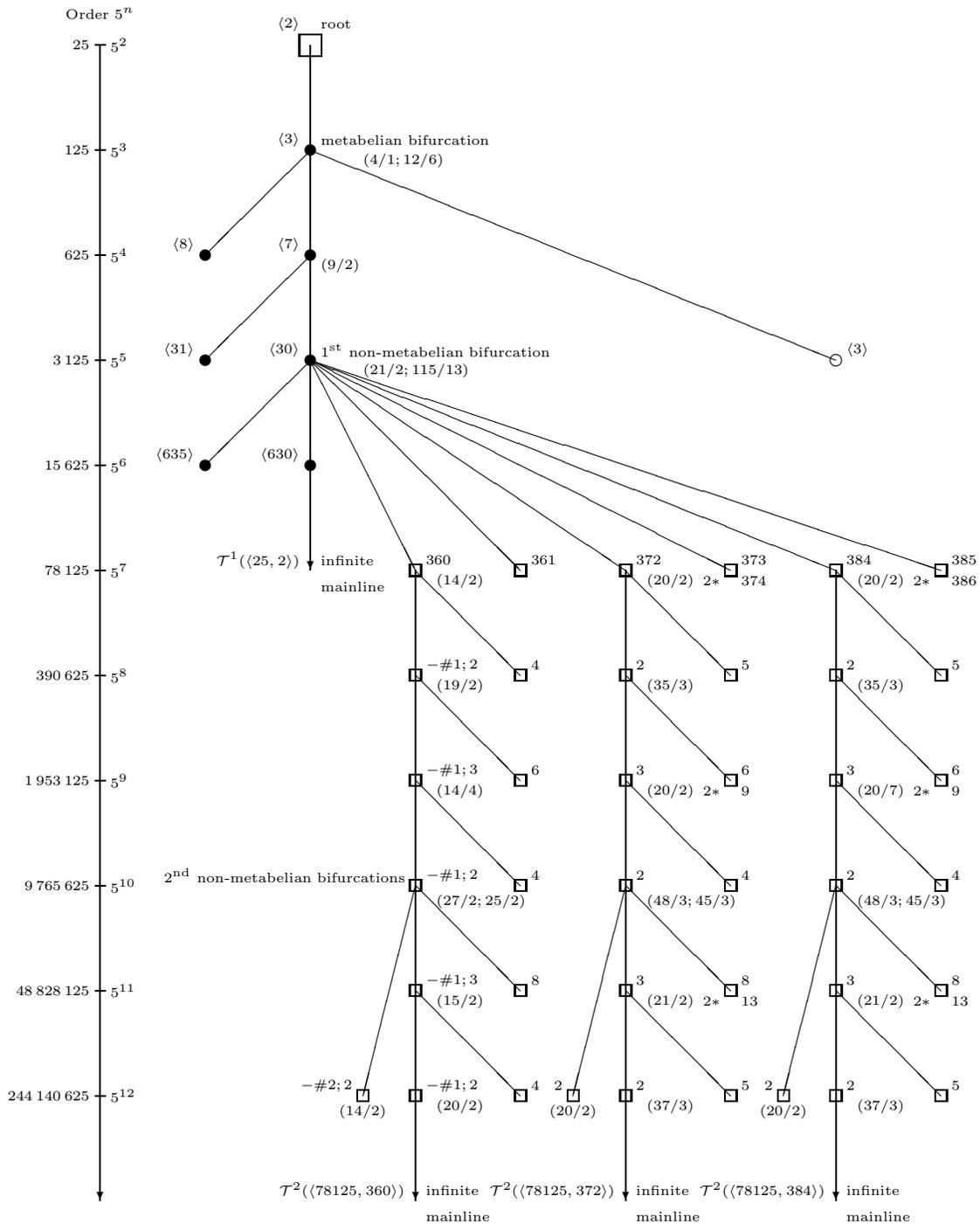}

\end{figure}

The diagram in Figure
\ref{fig:PrunedTree5x5}
visualizes the tree of \(5\)-groups \(G\)
with abelianization \(G/G^\prime\) of type \(\lbrack 5,5\rbrack\)
up to order \(5^{12}\),
\textit{pruned} by restriction to the transfer kernel types
a.1, \(\varkappa=(0,0,0,0,0,0)\), of vertices with defect \(k=0\) on infinite paths,
and a.2, \(\varkappa\sim (1,0,0,0,0,0)\), of leaves.

Groups of coclass \(\mathrm{cc}(G)=1\) are metabelian,
and groups of coclass \(\mathrm{cc}(G)=2\), except \(\langle 3125,3\rangle\), have derived length \(\mathrm{dl}(G)=3\).
The big contour square {\large \(\square\)} denotes the abelian root of the tree,
full discs {\large \(\bullet\)} denote metabelian groups with cyclic centre of order \(5\),
a contour circle {\large \(\circ\)} denotes a metabelian group with bicyclic centre of type \((5,5)\),
and smaller contour squares {\small \(\square\)} denote non-metabelian groups.
The symbol \(2^\ast\) means a batch of two siblings with common parent, where only one vertex is drawn

Up to order \(5^7\), vertices are labelled by their SmallGroups Library identifiers
\cite{BEO1,BEO2}
in angle brackets, where we omit the order, which is given on the left hand scale.
Starting with order \(5^8\), relative identifiers in the form \(-\#s;n\),
with step size \(s\) and counter \(n\), are used
as in the output produced by the ANUPQ package
\cite{GNO}
of GAP
\cite{GAP}
and MAGMA
\cite{MAGMA}.
For additional ease of identification,
we give the descendant numbers of some relevant capable vertices in the format
\(\left(N_1/C_1;\ldots;N_\nu/C_\nu\right)\) as described in
\cite[(34), p.180]{Ma6},
where \(\nu\) denotes the nuclear rank.
 


\begin{figure}[ht]
\caption{Normal lattice of the \(5\)-tower groups 
\(\langle 78125,n\rangle\) \((n\in\lbrace 361, 373, 374, 385, 386\rbrace)\)}
\label{fig:5TwrGrp3812377}


{\tiny

\setlength{\unitlength}{0.9cm}
\begin{picture}(15,17)(-5,2)

\put(-4,18.3){\makebox(0,0)[cb]{order \(5^n\)}}
\put(-4,16){\vector(0,1){2}}
\put(-4.2,16){\makebox(0,0)[rc]{\(78\,125\)}}
\put(-3.8,16){\makebox(0,0)[lc]{\(5^7\)}}
\put(-4.2,14){\makebox(0,0)[rc]{\(15\,625\)}}
\put(-3.8,14){\makebox(0,0)[lc]{\(5^6\)}}
\put(-4.2,12){\makebox(0,0)[rc]{\(3\,125\)}}
\put(-3.8,12){\makebox(0,0)[lc]{\(5^5\)}}
\put(-4.2,10){\makebox(0,0)[rc]{\(625\)}}
\put(-3.8,10){\makebox(0,0)[lc]{\(5^4\)}}
\put(-4.2,8){\makebox(0,0)[rc]{\(125\)}}
\put(-3.8,8){\makebox(0,0)[lc]{\(5^3\)}}
\put(-4.2,6){\makebox(0,0)[rc]{\(25\)}}
\put(-3.8,6){\makebox(0,0)[lc]{\(5^2\)}}
\put(-4.2,4){\makebox(0,0)[rc]{\(5\)}}
\put(-4.2,2){\makebox(0,0)[rc]{\(1\)}}
\multiput(-4.1,2)(0,2){8}{\line(1,0){0.2}}
\put(-4,2){\line(0,1){14}}

\multiput(8,16)(0,-4){2}{\line(1,0){2}}
\put(9,14.5){\vector(0,1){1.5}}
\put(9,14.25){\makebox(0,0)[cc]{first}}
\put(9,13.75){\makebox(0,0)[cc]{stage}}
\put(9,13.5){\vector(0,-1){1.5}}
\put(9,8.5){\vector(0,1){3.5}}
\put(9,8.25){\makebox(0,0)[cc]{second}}
\put(9,7.75){\makebox(0,0)[cc]{stage}}
\put(9,7.5){\vector(0,-1){3.5}}
\put(9,3.5){\vector(0,1){0.5}}
\put(9,3.25){\makebox(0,0)[cc]{third}}
\put(9,2.75){\makebox(0,0)[cc]{stage}}
\put(9,2.5){\vector(0,-1){0.5}}
\multiput(8,2)(0,2){2}{\line(1,0){2}}

\put(0.3,2){\makebox(0,0)[rc]{\(\zeta_0(G)\)}}
\put(0.5,2){\circle*{0.2}}
\put(0.7,2){\makebox(0,0)[lc]{\(\gamma_6(G)=1\)}}

\put(-1.2,4){\makebox(0,0)[rb]{\(G^{\prime\prime}\)}}
\put(-1.2,3.8){\makebox(0,0)[rt]{\(u_5\)}}
\put(-1,4){\circle*{0.2}}
\multiput(-0.4,4)(0.6,0){5}{\circle*{0.1}}
\put(2.2,4){\makebox(0,0)[lc]{\(s_5\)}}

\multiput(0.5,6)(-1.5,-2){2}{\line(3,-4){1.5}}
\multiput(0.5,6)(-0.9,-2.1){2}{\line(1,-2){0.9}}
\multiput(0.5,6)(-0.3,-2.1){2}{\line(1,-6){0.3}}
\multiput(0.5,6)(0.3,-2.1){2}{\line(-1,-6){0.3}}
\multiput(0.5,6)(0.9,-2.1){2}{\line(-1,-2){0.9}}
\multiput(0.5,6)(1.5,-2){2}{\line(-3,-4){1.5}}

\put(0.3,6){\makebox(0,0)[rc]{\(\zeta_1(G)\)}}
\put(0.5,6){\circle*{0.1}}
\put(0.7,6){\makebox(0,0)[lc]{\(\gamma_5(G)\)}}

\put(1.8,8){\makebox(0,0)[rc]{\(\zeta_2(G)\)}}
\put(2,8){\circle*{0.1}}
\put(2.2,8){\makebox(0,0)[lb]{\(\gamma_4(G)\)}}
\put(2.2,7.8){\makebox(0,0)[lt]{\(s_4\)}}

\put(2,8){\line(-3,-4){1.5}}

\put(3.3,10){\makebox(0,0)[rc]{\(\zeta_3(G)\)}}
\put(3.5,10){\circle*{0.1}}
\put(3.7,10){\makebox(0,0)[lb]{\(\gamma_3(G)\)}}
\put(3.7,9.8){\makebox(0,0)[lt]{\(s_3\)}}

\put(3.5,10){\line(-3,-4){1.5}}

\put(4.8,12){\makebox(0,0)[rc]{\(\zeta_4(G)\)}}
\put(5,12){\circle*{0.2}}
\put(5.2,12){\makebox(0,0)[lb]{\(\gamma_2(G)=G^\prime\)}}
\put(5.2,11.8){\makebox(0,0)[lt]{\(s_2\)}}

\put(5,12){\line(-3,-4){1.5}}

\put(3.3,14.1){\makebox(0,0)[rb]{\(M_1\)}}
\put(3.3,13.8){\makebox(0,0)[rt]{\(y\)}}
\put(4.1,14.1){\makebox(0,0)[cb]{\(M_6\)}}
\put(4.7,14.1){\makebox(0,0)[cb]{\(M_5\)}}
\multiput(3.5,14)(0.6,0){6}{\circle*{0.1}}
\put(5.3,14.1){\makebox(0,0)[cb]{\(M_4\)}}
\put(5.9,14.1){\makebox(0,0)[cb]{\(M_3\)}}
\put(6.7,14.1){\makebox(0,0)[lb]{\(M_2\)}}
\put(6.7,13.8){\makebox(0,0)[lt]{\(x\)}}

\multiput(5,16)(-1.5,-2){2}{\line(3,-4){1.5}}
\multiput(5,16)(-0.9,-2.1){2}{\line(1,-2){0.9}}
\multiput(5,16)(-0.3,-2.1){2}{\line(1,-6){0.3}}
\multiput(5,16)(0.3,-2.1){2}{\line(-1,-6){0.3}}
\multiput(5,16)(0.9,-2.1){2}{\line(-1,-2){0.9}}
\multiput(5,16)(1.5,-2){2}{\line(-3,-4){1.5}}
\put(4.8,16){\makebox(0,0)[rc]{\(\zeta_5(G)\)}}
\put(5,16){\circle*{0.2}}
\put(5.2,16){\makebox(0,0)[lc]{\(\gamma_1(G)=G\)}}

\end{picture}

}

\end{figure}
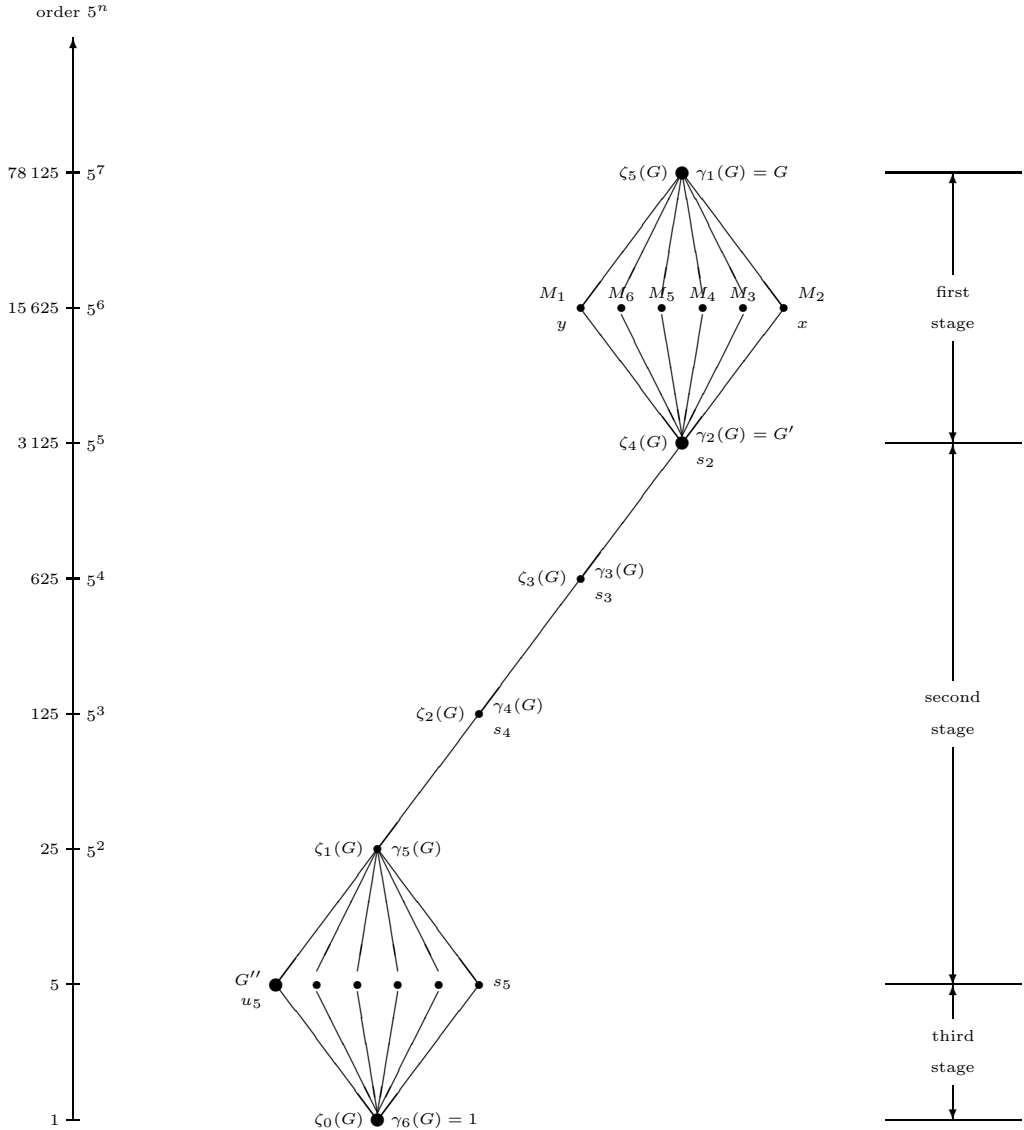

The diagram in Figure
\ref{fig:5TwrGrp3812377}
visualizes the lattice of all normal subgroups of the \(5\)-class tower groups
\(G_{n}:=\langle 78125,n\rangle\) with \(n\in\lbrace 361,373,374,385,386\rbrace\).
Polycyclic power commutator presentations of these groups are given by

\begin{equation}
\label{eqn:pcPresentations}
\begin{aligned}
G_{n}= & \langle\ x,y,s_2,s_3,s_4,s_5,u_5\ \mid \\
        & s_2=\lbrack y,x\rbrack,\ s_3=\lbrack s_2,x\rbrack,\ s_4=\lbrack s_3,x\rbrack,\ s_5=\lbrack s_4,x\rbrack, \\
        & x^5=W,\ y^5=s_5^4,\ u_5=\lbrack s_3,y\rbrack=\lbrack s_4,y\rbrack,\ \lbrack s_3,s_2\rbrack=u_5^4\ \rangle,
\end{aligned}
\end{equation}

\noindent
where the relator word \(W\) for the fifth power of the first generator \(x\) is decisive:

\begin{equation}
\label{eqn:RelatorWord}
W=
\begin{cases}
s_5        & \text{ if } n=361, \\
s_5u_5     & \text{ if } n=373, \\
s_5^2u_5   & \text{ if } n=374, \\
s_5u_5^2   & \text{ if } n=385, \\
s_5^2u_5^2 & \text{ if } n=386.
\end{cases}
\end{equation}

\noindent
The presentation for the second derived quotient of \(G_n\) is obtained by putting \(u_5:=1\).
Both resulting relator words \(W=s_5\) and \(W=s_5^2\) yield the same metabelian group
\(\langle 15625,635\rangle\).



\section{Acknowledgements}
\label{s:Acknowledgements}

\noindent
The author gratefully acknowledges that his research is supported
by the Austrian Science Fund (FWF): P 26008-N25.




\end{document}